\documentclass[10pt,epsfig,latexsym,amsfonts,twoside]{article}
\usepackage{amssymb,amsmath,amsthm,amscd,epsfig}
\usepackage[backref]{hyperref}
\usepackage[normalem]{ulem}

\newcommand{\defi}[1]{\textsf{#1}} 

\def\part#1{\frac{\partial\phantom{#1}}{\partial#1}}
\newtheorem{thm}{Theorem}
\newtheorem{prp}[thm]{Proposition}
\newtheorem{lem}[thm]{Lemma}
\newtheorem{cor}[thm]{Corollary}

\newenvironment{prf}{\begin{trivlist}\item[]{\bf Proof} }
{\hfill $\Box$ \end{trivlist}}

\newenvironment{rmk}{\begin{trivlist}\item[]{\bf Remark} }
{\end{trivlist}}
\newenvironment{que}{\begin{trivlist}\item[]{\bf Question} }
{\end{trivlist}}

\newcommand{\ca}{c_\alpha}
\newcommand{\ia}{i_\alpha}

\newcommand{\disc}{d\left(\Ga\right)}

\newcommand{\hideqed}{\renewcommand{\Box}{}} 

\usepackage[usenames,dvipsnames]{color}

\def\Z{\ifmmode{{\mathbb Z}}\else{${\mathbb Z}$}\fi}
\def\Q{\ifmmode{{\mathbb Q}}\else{${\mathbb Q}$}\fi}
\def\C{\ifmmode{{\mathbb C}}\else{${\mathbb C}$}\fi}
\def\P{\ifmmode{{\mathbb P}}\else{${\mathbb P}$}\fi}
\def\H{\ifmmode{{\mathrm H}}\else{${\mathrm H}$}\fi}
\def\G{\ifmmode{{\mathbb G}}\else{${\mathbb G}$}\fi}
\def\R{\ifmmode{{\mathbb R}}\else{${\mathbb R}$}\fi}
\def\F{\ifmmode{{\mathbb F}}\else{${\mathbb F}$}\fi}
\def\O{\ifmmode{{\cal O}}\else{${\cal O}$}\fi}
\def\Br{\ifmmode{{\mathrm{Br}}}\else{${\mathrm{Br}}$}\fi}
\def\D{\ifmmode{{\cal{D}}^b}\else{${{\cal{D}}^b}$}\fi}

\newcommand{\Sbar}{{\overline{S}}}
\newcommand{\Qbar}{{\overline{\Q}}}

\newcommand{\kbar}{{\overline{k}}}

\newcommand{\Ga}{\Gamma_{\langle\alpha\rangle}}

\newcommand{\Adeles}{{\mathbb A}}

\newcommand{\calC}{{\mathcal C}}
\newcommand{\calD}{{\mathcal D}}

\newcommand{\calK}{{\mathcal K}}

\newcommand{\calS}{{\mathcal S}}

\DeclareMathOperator{\inv}{inv}

\DeclareMathOperator{\im}{im}

\DeclareMathOperator{\Hom}{Hom}

\DeclareMathOperator{\Aut}{Aut}
\DeclareMathOperator{\Gal}{Gal}

\DeclareMathOperator{\Pic}{Pic}

\DeclareMathOperator{\Spec}{Spec}

\DeclareMathOperator{\Proj}{Proj}

\DeclareMathOperator{\ev}{ev}

\DeclareMathOperator{\et}{et}

\DeclareMathOperator{\GL}{GL}

\DeclareMathOperator{\Sign}{Sign}
\DeclareMathOperator{\NS}{NS}

\begin{document}

\title{Brauer groups on K3 surfaces and arithmetic applications}
\author{Kelly McKinnie, Justin Sawon, Sho Tanimoto, and Anthony V\'arilly-Alvarado}
\date{April, 2014}
\maketitle

\begin{abstract}

For a prime $p$, we study subgroups of order $p$ of the Brauer group $\Br(S)$ of a general complex polarized K3 surface of degree $2d$, generalizing earlier work of van Geemen. These groups correspond to sublattices of index $p$ of the transcendental lattice $T_S$ of $S$; we classify these lattices up to isomorphism using Nikulin's discriminant form technique. We then study geometric realizations of $p$-torsion Brauer elements as Brauer-Severi varieties in a few cases via projective duality. We use one of these constructions for an arithmetic application, giving new kinds of counter-examples to weak approximation on K3 surfaces of degree two, accounted for by transcendental Brauer-Manin obstructions.

\end{abstract}

\section{Introduction}

Let $S$ be a smooth projective geometrically integral variety over a number field $k$; write $\Adeles$ for the ring of adeles of $k$. Assume that $S(\Adeles) \neq \emptyset$. It is well-known that elements of the Brauer group $\Br(S) := \H^2_{\et}(S,\G_m)$ can obstruct the existence of $k$-rational points of $S$~\cite{manin}; in such cases we say there is a \defi{Brauer-Manin obstruction} to the \defi{Hasse principle}. Brauer elements can also explain why sometimes the image of the diagonal map $S(k) \hookrightarrow S(\Adeles)$ fails to be dense; in such cases we say there is a \defi{Brauer-Manin obstruction} to \defi{weak approximation} (see Section~\ref{S:BMobs} for more details).

In order to show that a particular Brauer element obstructs the Hasse principle or weak approximation, one often needs a geometric realization of the Brauer element, especially if the element remains non-trivial after extension of scalars to an algebraic closure (such elements are known as \defi{transcendental} elements). In~\cite{vangeemen05} van Geemen studied transcendental $2$-torsion Brauer elements on generic complex K3 surfaces $S$ of degree two. Using lattice-theoretic methods, he gave a classification into three cases, and described geometric realizations of the Brauer elements. In the first case, $S$ is the double cover of $\P^2$ branched over the sextic discriminant curve of the quadric surface fibration on a cubic fourfold containing a plane; Hassett, V{\'a}rilly-Alvarado, and Varilly~\cite{hvv11} showed that the Brauer element in this example can obstruct weak approximation. The second case involves a double cover of $\P^2\times\P^2$ branched over a hypersurface of bidegree $(2,2)$; Hassett and V{\'a}rilly-Alvarado~\cite{hv11} showed that the Brauer element in this example can obstruct the Hasse principle. In the third case, $S$ is the double cover of $\P^2$ branched over the sextic discriminant curve of the net of quadrics defining a K3 surface of degree eight in $\P^5$; transcendental Brauer elements of this type have not yet been used for arithmetic applications.

The goal of this paper is to extend this earlier work in several directions: \\

\noindent {\bf Classification of order $p$ subgroups of general K3 surfaces.}
Let $S$ be a general complex polarized K3 surface of degree $2d$, and write $T_S = \NS(S)^\perp \subset \H^2(S,\Z)$ for its \defi{transcendental lattice}. Let $p$ be an odd prime.  To classify subgroups of order $p$ in $\Br(S)$ we use the correspondence
\[
\{ \textup{subgroups of order $p$ in }\Br(S) \} \longleftrightarrow
\{ \textup{sublattices of index $p$ in }T_S\}
\]
furnished by the exponential sequence and elementary lattice-theoretic properties of $\H^2(S,\Z)$ (see Section~2 of~\cite{vangeemen05}). We apply Nikulin's discriminant form technique~\cite{nikulin} to classify sublattices of index $p$ in $T_S$ up to isomorphism, and we count the number of lattices in each isomorphism class. This is the content of Sections~\ref{ss: set-up}--\ref{ss: summary}. Our main result in this direction is Theorem~\ref{thm: lattices clean}, showing there are three or four classes of $p$-torsion subgroups of $\Br(S)$, according to whether $p\nmid d$ or not, respectively.  We expect that each class of subgroups is associated to a geometric construction for $p$-torsion Brauer elements, like in the case of $2$-torsion.  Indeed, the lattice theory already suggests a strong connection between certain $p$-torsion classes on K3 surfaces of degree two and higher degree K3 surfaces or special cubic fourfolds. We explore these connections in Sections~\ref{ss: mukai duals} and~\ref{ss: special 4-folds} following Mukai~\cite{mukai87} and building on Hassett~\cite{h00}, respectively.  \\

\noindent {\bf Geometric realization of Mukai dualities.}
Having classified $p$-torsion elements of the Brauer group, we next look for geometric realizations as Brauer-Severi varieties.  In the third case of van Geemen's analysis of $2$-torsion Brauer elements on degree two K3 surfaces, the Brauer element on $S$ comes from the Fano variety of maximal isotropic subspaces inside the quadrics defining the associated degree eight K3 surface $X$. We describe how $S$ can also be interpreted as a Mukai moduli space of stable sheaves on $X$ (Lemma~\ref{lem: moduli}) and how the Brauer element is the obstruction to the fineness of this moduli space (Lemma~\ref{obstruction}). This example then admits a vast generalization: given a K3 surface $X$ of degree $2dp^2$, there exists a `Mukai dual' K3 surface $S$ given by a moduli space of stable sheaves on $X$, and a $p$-torsion Brauer element on $S$ obstructing the existence of a universal sheaf. In some low degree cases, including the case $d=1$ and $p=2$ above, this Mukai duality can be realized as projective duality, and the Brauer element can be realized geometrically as a Brauer-Severi variety. In Section~\ref{ss: 18to2} we describe the case $d=1$ and $p=3$, showing that the $3$-torsion Brauer element on $S$ comes from the Fano variety of cubic surface scrolls inside a certain net of Fano fourfolds containing $X$. In Section~\ref{ss: 16to4} we describe the case $d=2$ and $p=2$ as an instance of projective duality, though we leave as an open question the geometric realization of the resulting $2$-torsion Brauer element.  The Mukai dualities we discuss have been studied before~\cite{mukai87,ik13,kr10,ir05,ir07}, although we hope that our exposition will be useful to arithmetic geometers. \\

\noindent {\bf Explicit obstructions to weak approximation.}
Returning to the third case in van Geemen's classification of $2$-torsion Brauer elements on K3 surfaces of degree two, we construct an explicit K3 surface $S$ of degree two, together with a transcendental $2$-torsion element $\alpha \in \Br(S)$ that obstructs weak approximation; see Theorem~\ref{thm:counterexample}. We are able to compute an obstruction by interpreting $\alpha$ as the even Clifford algebra of the discriminant cover $S \to \P^2$ of a given net of quadrics in $\P^5$, following Auel, Bernardara and Bolognesi~\cite{abb11}. We use elementary properties of Clifford algebras to represent $\alpha$ as a product of two quaternion algebras over the function field $k(S)$. Along the way we prove a curious result (Corollary~\ref{cor:continuity}), which explains why it may be difficult to use elements of the form $\alpha$ to obstruct the Hasse principle (see Section~\ref{ss: hp?} as well).

It would be interesting to use the construction of Section~\ref{ss: 18to2} to build a K3 surface $S$ (of degree $2$) with an obstruction to the Hasse principle arising from a $3$-torsion in $\Br(S)$.  At present, no examples like this are known to exists, and recent work of Ieronymou and Skorobogatov naturally raises this problem; see the discussion after Corollary~1.3 of~\cite{is13}.\\

\noindent {\bf Acknowledgements.} The authors thank Asher Auel, Brendan Hassett, Danny Krashen, Alexander Kuznetsov, Sukhendu Mehrotra, and Ronald van Luijk for several discussions on this work. We are grateful to the referees for their careful review of the manuscript; their suggestions greatly improved the exposition of the paper. We also thank the American Institute for Mathematics, Palo Alto, for hosting the workshop ``Brauer groups and obstruction problems: moduli spaces and arithmetic'' and for travel funding. The second and fourth authors were supported by the NSF under grant numbers DMS-1206309 and DMS-1103659/DMS-1352291, respectively.\\

\section{Lattice gymnastics}

Let $S$ be a complex, projective K3 surface with N\'eron-Severi group $\mathrm{NS}(S)$ isomorphic to $\Z h$.  The intersection form makes the singular cohomology group ${\rm H}^2(S,\Z)$ into a lattice.  Write $T_S = \mathrm{NS}(S)^\perp$ for the transcendental lattice of $S$, and let $p$ be a prime number. By \S2.1 of~\cite{vangeemen05}, a nontrivial element $\alpha \in \Br(S)[p]$ gives rise to a surjective homomorphism $\alpha\colon T_S \to \Z/p\Z$.  The kernel of this homomorphism is a sublattice of $T_S$ of index $p$.  Conversely, a sublattice $\Gamma\subset T_S$ of index $p$ determines a subgroup $\langle \alpha\rangle \subseteq \Br(S)$ of order $p$.  Accordingly, we write $\Gamma_{\langle\alpha\rangle} = \Gamma$ for such a sublattice. In \S9 of~\cite{vangeemen05}, van Geemen classifies the isomorphism types of sublattices of index $2$ in $T_S$. He shows that there are three isomorphism types, and for each type he offers an auxiliary variety, together with a geometric construction that takes the auxiliary variety and recovers the original K3 surface $S$ together with a Brauer-Severi bundle over $S$ corresponding to the unique nontrivial element of $\langle\alpha\rangle \subset \Br(S)[2]$.

One might hope that for odd $p$, each isomorphism type of $\Gamma_{\langle\alpha\rangle}$ has an associated geometric construction that could be used for arithmetic applications.  Thus, it is of interest to classify sublattices $\Gamma_{\langle\alpha\rangle}\subset T_S$ of odd prime index $p$ up to isomorphism.  Our strategy is to generalize the proofs of Propositions~3.3 and~9.2 in~\cite{vangeemen05}.

\subsection{Set-up}
\label{ss: set-up}

Let $\left(\Gamma,(\cdot\,,\cdot)\right)$ be a \defi{lattice}, i.e., a free $\Z$-module of finite rank, together with a nondegenerate integral symmetric bilinear form $(\cdot\,,\cdot)$. We write $O(\Gamma)$ for the group of orthogonal transformations of $\Gamma$. Denote by $\Gamma^*$ the dual lattice $\Hom(\Gamma,\Z)$; the bilinear form on $\Gamma$ can be extended $\Q$-bilinearly to $\Gamma^*$. We embed $\Gamma \subseteq \Gamma^*$ via the map
\[
\gamma \mapsto [\phi_\gamma\colon \Gamma \to \Z,\ \ \delta \mapsto (\delta,\gamma)]
\]
The \defi{discriminant group} $d(\Gamma)$ of $\Gamma$ is $\Gamma^*/\Gamma$; it is a finite abelian group whose order is the \defi{discriminant} of $\Gamma$. A lattice is \defi{unimodular} if its discriminant group is trivial. If $\Gamma$ is an \defi{even lattice}, i.e., $(\gamma,\gamma) \in 2\Z$ for all $\gamma \in \Gamma$, then there is a quadratic form
\[
q\colon d(\Gamma) \to \Q/2\Z \qquad x + \Gamma \mapsto (x,x) \bmod 2\Z,
\]
called the \defi{discriminant form} of $\Gamma$. One also obtains a symmetric bilinear form 
\[
b\colon d(\Gamma) \times d(\Gamma) \to \Q/\Z,
\]
which is characterized by the identity
\[
q(x + y) - q(x) - q(y) \equiv 2b(x,y) \bmod 2\Z.
\]
Nikulin showed in~Corollary 1.13.3 of \cite{nikulin} that an even indefinite lattice whose rank exceeds (by at least 2) the minimal number of generators of $d(\Gamma)$ is determined by its rank, signature and discriminant quadratic form. We will use this fact in what follows, without explicitly mentioning it every time.  

We write $d(\Gamma)_p$ for the $p$-Sylow subgroup of $d(\Gamma)$, and $q_p$ for the restriction of $q$ to this subgroup. By Proposition~1.2.2 of~\cite{nikulin} there is an orthogonal decomposition $q = \bigoplus_p q_p$ as $p$ runs over prime numbers dividing the order of $d(\Gamma)$.

Let $S$ be a complex projective K3 surface. By \S1 of~\cite{torelli}, we can write
\[
{H}^2(S , \Z) \cong U_1 \oplus U_2 \oplus U_3 \oplus E_8(-1)^2  =: \Lambda_{\mathrm K3},
\]
where the $U_i$ are hyperbolic planes (i.e., even unimodular lattices of signature $(1,1)$), and $E_8(-1)$ is the unique negative definite even unimodular lattice of rank eight. In particular, $\Lambda_{\mathrm K3}$ is even, unimodular, and has signature $(3,19)$. If $\mathrm{NS}(S) = \Z h$, with $h^2 = 2d > 0$, then by Theorem~1.1.2 of~\cite{nikulin}, the inclusion $\mathrm{NS}(S)=\Z h \hookrightarrow \Lambda_{\mathrm K3}$ is unique up to isometry, and therefore we may assume that
\[
\mathrm{NS}(S) = \Z h \cong \Z(1,d)\hookrightarrow U_1\hookrightarrow U_1\oplus \Lambda'=\Lambda_{\mathrm K3},
\]
where $\Lambda'=U_2\oplus U_3\oplus E_8(-1)^2$.  Let $v=(1,-d)\in U_1$, so that $v^2=-2d$.  Then
\[
T_S\cong \Z v\oplus \Lambda' \cong \langle -2d\rangle\oplus \Lambda'.
\]

\subsection{Discriminant groups of $p$-torsion Brauer classes}

We begin by analyzing the homomorphism $\alpha:T_S \to {\Z}/p\Z$ associated to a nonzero element $\alpha \in \Br(S)[p]$.  Since the lattice $\Lambda'$ is unimodular, and hence self-dual, there is a $\lambda_\alpha \in \Lambda'$, whose class in $\Lambda'/p\Lambda'$ is uniquely determined, such that the homomorphism $\alpha$ can be expressed as
\begin{eqnarray*}
\alpha:T_S &\to& {\Z}/p\Z, \\
zv+\lambda'&\mapsto& z\ia+\langle \lambda_\alpha,\lambda'\rangle \bmod p,
\end{eqnarray*}
for some integer $\ia$, which we may assume is in the range $0\leq \ia \leq p-1$.
If $\ia\neq 0$, we write $\ia^{-1}$ for the inverse of $\ia$ modulo $p$ in the range $1\leq \ia^{-1} \leq p-1$. Since $\alpha$ and $\ia^{-1}\alpha$ have the same kernel $\Ga$, and since the kernel determines the subgroup $\langle\alpha\rangle \subseteq \Br(S)[p]$, replacing $\alpha$ with $\ia^{-1}\alpha$, we may assume that $\ia = 1$.

Define $\ca \in \Z$ by $\lambda_\alpha^2={-}2\ca$; the class of $\ca$ modulo $p$ is uniquely determined by $\alpha$.  The lattice $\Lambda'$ is even, unimodular and has signature $(2,18)$.  Applying Theorem~1.1.2 of \cite{nikulin}, we conclude that any embedding $\lambda_\alpha\Z \hookrightarrow \Lambda'$ is unique up to isometry.  Therefore, without loss of generality, we  assume that
\[
\lambda_\alpha=(1,-c_\alpha)\in U_2\hookrightarrow U_2\oplus U_3\oplus E_8(-1)^2=\Lambda'.
\]
Let $\Lambda''=U_3\oplus E_8(-1)^2\subset \Lambda'$, and let $\Ga=\ker(\alpha)$.  We compute  
\begin{eqnarray*}
\Ga&=&\{zv+(a,b)+\lambda''\mid z \in \Z, (a,b)\in U_2, \lambda'' \in \Lambda'', \textrm{ and }z\ia-a\ca+b\equiv 0\bmod p\}\\
&=&\{zv+(a,kp-z\ia+a\ca)+\lambda'' \mid z,a,k \in \Z, \lambda'' \in \Lambda''\}\\
&=&\{z(v+(0,-\ia))+a(1,\ca)+k(0,p)+\lambda''\mid z,a,k \in \Z, \lambda'' \in \Lambda''\}.
\end{eqnarray*}
Hence, for fixed $d$ and $p$, the lattice $\Ga$ (and by extension, its discriminant form), is determined by the values $\ia$ and $\ca$ modulo $p$.  Let $M_\alpha$ be the rank three lattice with Gram matrix 
\[
\left(\begin{array}{ccc}-2d&-\ia&0\\-\ia&2\ca&p\\0&p&0\end{array}\right).
\]
Our computation shows that $\Ga \cong M_\alpha \oplus \Lambda''$.  Unimodularity of $\Lambda''$ implies that $\Ga$ and $M_\alpha$ have isomorphic discriminant groups and isometric discriminant forms.  When we need to, we will set $\Ga=\Gamma_{\ia,\ca}$.

\begin{thm}
\label{thm: lattices}
Let $S$ be a complex projective K3 surface of degree $2d$ such that $\NS(S)\cong \Z h$, and write $T_S := \NS(S)^\perp \subset {\rm H}^2(S,\Z)$ for its transcendental lattice. Let $p$ be a prime number, and let $\alpha \in \Br(S)[p]$ be nontrivial, with associated index $p$ sublattice $\Ga=\Gamma_{\ia,\ca} \subset T_S$.
\begin{enumerate}
\item  If $\ia =0$, then
\[\disc \cong \left\{\begin{array}{ll}
				{\Z} /2d\Z\oplus \Z /p^2\Z & \textrm{ if } p\nmid\ca \textrm{ and $p$ is odd},\\
				{\Z} /2d\Z\oplus \Z /p\Z\oplus \Z /p\Z & \textrm{ if } p\mid\ca \textrm{ or } p = 2.
			\end{array}
			\right.\]

\item If $\ia = 1$, then
\[\disc \cong \left\{\begin{array}{ll}
				{\Z} /2dp^2\Z & \textrm{ if } p\nmid(1+ 4\ca d), \\
				{\Z} /2dp\Z\oplus \Z /p\Z & \textrm{ if } p\mid(1+ 4\ca d).
			\end{array}
			\right.\]
\end{enumerate}
\end{thm}

\begin{proof} Each element in $\disc\cong d(M_\alpha) = M_\alpha^*/M_\alpha$ is represented by a $\gamma \in M^*_\alpha$ satisfying $2dp^2\gamma\in M_\alpha$.  In other words, we can write $\gamma=\frac{\gamma'}{2dp^2}$, with $\gamma' \in M_\alpha$. Therefore, we represent elements of $d(M_\alpha)$ as triples $\gamma'=(A,B,C) \in M_\alpha$ such that $\left( (A,B,C),(x,y,z)\right)\in 2dp^2\Z$ for all $x,y,z \in \Z$. Taking $(x,y,z)$ to be $(1,0,0)$, $(0,1,0)$ and $(0,0,1)$, in turn, we see this happens if and only if there exist some $k_0$, $k_1$, $k_2 \in \Z$ such that
\begin{eqnarray*}
A&=&-p(\ia k_2{+}pk_0)\\
B&=&2dpk_2\\
C&=&2dpk_1-\ia^2k_2-p\ia k_0-4\ca dk_2
\end{eqnarray*}

 \noindent
 {\bf Case 1:} $\ia =0$.  In this case the equations for $A,\,B,\,C$ reduce to
 \begin{eqnarray*}
A&=&{-}p^2k_0\\
B&=&2dpk_2\\
C&=&2dpk_1-4\ca dk_2
\end{eqnarray*}
 for some $k_0,k_1,k_2 \in \Z$.  In particular, the triples
 \[v_1:=\frac{(p^2,0,0)}{2dp^2},\,\,v_2:=\frac{(0,2dp,-4\ca d)}{2dp^2} \textrm{ and } v_3:=\frac{(0,0,2dp)}{2dp^2}\]
 represent non-trivial elements of $d(M_\alpha)$.  If $p$ is odd, then the elements $v_1$ and $v_2$ generate subgroups of respective orders $2d$ and $p^2/\gcd(\ca,p)$, and these subgroups intersect trivially.  If $p\nmid\ca$, this shows that $d(M_\alpha) \cong \Z/2d\Z\oplus \Z/p^2\Z$.  If $p\mid\ca$ then $v_1$, $v_2$ and $v_3$ are independent elements that generate subgroups of respective orders $2d$, $p$ and $p$, showing that $d(M_\alpha) \cong \Z/2d\Z\oplus \Z/p\Z\oplus \Z/p\Z$. If $p = 2$, then $v_2$ has order two, and thus $\disc \cong \Z/2d\Z\oplus \Z/2\Z\oplus \Z/2\Z$.

 \smallskip
 \noindent
 {\bf Case 2:} $\ia = 1$.  Let $k_2 = {-1}$, and $k_0 = k_1 = 0$. Then
 \[v_4:=\frac{(p,-2dp,1+4\ca d)}{2dp^2}\]
is an element of $d(M_\alpha)$. Because of its first component, $v_4$ generates a subgroup of order divisible by $2dp$, hence $d(M_\alpha)$ is isomorphic to either $\Z/2dp^2\Z$ or $\Z/2dp\Z\oplus \Z/p\Z$.  Therefore, if $p\nmid(1+4\ca d)$, then $(2dp)v_4$ is not trivial in $d(M_\alpha)$, and hence $d(M_\alpha) \cong \Z/2dp^2\Z$.  On the other hand, if $p\mid(1 +4\ca d)$ then $v_4$ and 
\[
v_5 := \frac{1}{2dp^2}(0,0,2dp)
\] 
generate subgroups, of order $2dp$ and $p$ respectively, which intersect trivially.  Therefore, $d(M_\alpha) \cong \Z/2dp\Z\oplus \Z/p\Z$.
 \end{proof}
 
\subsection{Isomorphism classes of lattices}

Our next task is to determine when the lattices appearing in Theorem~\ref{thm: lattices} are isomorphic.  We do this by comparing their discriminant forms.  We begin by comparing lattices with cylic discriminant groups.  Throughout this section, we retain the notation of Theorem~\ref{thm: lattices}.

\begin{prp}
\label{prp: cyclic case}
Let $p$ be an odd prime.  Consider the lattices $\Ga$ in Theorem~\ref{thm: lattices} for which $\disc \cong {\Z}/2dp^2{\Z}$ is a cyclic group, generated by an element $v$. Write $q\colon \disc \to \Q/2{\Z}$ for the discriminant form of $\Ga$.
\begin{enumerate}
\item[(i)] If $p\mid d$, then all such lattices are isomorphic.
\item[(ii)] If $p\nmid d$, then there are two isomorphism classes of lattices.  The isomorphism type of $\Ga$ depends only on whether $-2dp^2q(v)$ is a square modulo $p$ or not.
\end{enumerate}
\end{prp}

\begin{rmk}
The analogous proposition when $p = 2$ is handled by van Geemen in Proposition~9.2 of~\cite{vangeemen05}.
\end{rmk}

\begin{proof}
Suppose first that $p\mid d$. It follows from Theorem~\ref{thm: lattices} and its proof that $\ia = 1$, $p\nmid(1 + 4\ca d)$, and $v_4 = \frac{1}{2dp^2}(p,-2dp,1+4\ca d)$ is a generator for $d(M_\alpha) \cong \disc$. Then
\[
q(v_4) = -\frac{1}{2dp^2}(1 + 4\ca d),
\]
and so two lattices $\Ga = \Gamma_{1,c_\alpha}$ and $\Gamma_{\langle\alpha'\rangle} = \Gamma_{1,c_{\alpha'}}$ are isomorphic if there exists $x \in \left({\Z}/2dp^2{\Z}\right)^\times$ such that
\begin{equation}
\label{eq:congruence}
(1 + 4\ca d) \equiv x^2(1 + 4c_{\alpha'}d) \bmod{m},
\end{equation}
where $m = 4dp^2$.  Such an $x$ exists if and only if~\eqref{eq:congruence} has a solution when $m = p$.  Indeed, if the latter congruence has a solution, then so does~\eqref{eq:congruence} for all $m = p^n$ with $n > 1$, by Hensel's lemma, and if $q\neq p$ is an odd prime dividing $d$, then~\eqref{eq:congruence} has a solution for all $m = q^n$ with $n > 0$, again by Hensel's lemma (the case $n = 1$ being trivial since $q\mid d$). Finally, if $2\nmid d$, then~\eqref{eq:congruence} clearly has a solution when $m = 4$, and if on the other hand $2\mid d$, then~\eqref{eq:congruence} has a solution when $m = 8$, and thus for any $m= 2^n$ with $n > 2$, by another application of Hensel's lemma. Putting all this together using the Chinese remainder theorem, we obtain a solution for~\eqref{eq:congruence} for $m = 4dp^2$, as claimed.  Since $p\mid d$, it is clear that~\eqref{eq:congruence} has a solution when $m = p$. This proves (i).

Next, assume that $p\nmid d$, and let $\Ga = \Gamma_{\ia,c_\alpha}$ and $\Gamma_{\langle\alpha'\rangle} = \Gamma_{i_{\alpha'},c_{\alpha'}}$ be two lattices with cyclic discriminant group.  By~Theorem~\ref{thm: lattices} and its proof, we may assume that the discriminant group of $\Ga$  is generated by either $v_4$ or
\[
v_1 + v_2 = \frac{(p^2,2dp,-4\ca d)}{2dp^2},
\]
and likewise for $\Gamma_{\langle\alpha'\rangle}$. We computed $q(v_4)$ above; now note that
\[
q(v_1 + v_2) = -\frac{1}{2dp^2}(p^2 + 4\ca d).
\]
Thus, we see that $-2dp^2q(v_1 + v_2)$  and $-2dp^2q(v_4)$ are integers not divisible by $p$. Write $v$ and $v'$ for the generators of $\disc$ and $d(\Gamma_{\langle\alpha'\rangle})$, respectively. Then $\Ga$ and $\Gamma_{\langle\alpha'\rangle}$ are isomorphic if and only if the congruence
\begin{equation}
\label{eq:congruence2}
-2dp^2q(v) \equiv -2dp^2q(v')x^2 \bmod{m}
\end{equation}
has a solution when $m = 4dp^2$. Arguing as in (i), this is equivalent to~\eqref{eq:congruence2} having a solution when $m = p$.
\end{proof}

Suppose next that $p \nmid d$ and that the discriminant group of a lattice $\Ga$ is not cyclic.  By Theorem~\ref{thm: lattices}, we have $d(\Ga) \cong {\Z}/2d{\Z} \oplus {\Z}/p{\Z} \oplus {\Z}/p{\Z}$, and there are two possible lattices with this discriminant group, characterized by the value of $\ia$ and $c_\alpha$.  These two lattices are in fact isomorphic, as the following lemma shows.

\begin{lem}
\label{lem: p not div d}
Let $p$ be an odd prime such that $p \nmid d$. There is a unique lattice $\Ga$, up to isomorphism, whose discriminant group $\disc$ is not cyclic.  Moreover, in this case we have $\disc \cong {\Z}/2d{\Z} \oplus {\Z}/p{\Z}\oplus {\Z}/p{\Z}$.
\end{lem}

\begin{proof}
We have already shown that $\disc \cong {\Z}/2d{\Z} \oplus {\Z}/p{\Z}\oplus {\Z}/p{\Z}$. Let $\Gamma$ be the lattice $\Ga=\Gamma_{0,0}$ determined by $\ia = 0$ and $c_\alpha = 0$, and let $\Gamma'$  be the lattice $\Gamma_{\langle \alpha'\rangle}=\Gamma_{i_{\alpha'},c_{\alpha'}}$ determined by $i_{\alpha'} = 1$ and $c_{\alpha'}$ with $p\mid (1 + 4c_{\alpha'} d)$. Write $q$ and $q'$ for their respective discriminant quadratic forms.  We show that $q$ and $q'$ are isometric. Using the notation of the proof of Theorem~\ref{thm: lattices}, we may assume that
\[
d(\Gamma) = \langle v_1\rangle \oplus \langle v_2, v_3 \rangle\qquad\textrm{and}\qquad
d(\Gamma') = \langle pv_4 \rangle \oplus \langle 2dv_4, v_5\rangle.
\]
Recall that $v_1$ and $pv_4$ have order $2d$ in their respective discriminant groups, while $v_2$, $v_3$, $2dv_4$ and $v_5$ each have order $p$.  By Proposition~1.2.1 of~\cite{nikulin}, we know that
\[
q = q\big|_{{\Z}/2d{\Z}} \oplus q\big|_{d(\Gamma)_p}\qquad\textrm{and}\qquad
q' = q'\big|_{{\Z}/2d{\Z}} \oplus q'\big|_{d(\Gamma')_p}
\]
Thus, to prove that $q$ and $q'$ are isometric, it suffices to exhibit an $x \in ({{\Z}/2d{\Z}})^\times$ such that
\[
q(xv_1) \equiv q'(pv_4) \pmod {2{\Z}},
\]
and an isomorphism $\phi\colon d(\Gamma')_p \xrightarrow{\sim} d(\Gamma)_p$ of ${\Z}/p{\Z}$-vector spaces such that
\begin{equation}
\label{eq: pSylow}
q(\phi(v)) \equiv q'(v) \pmod {2{\Z}}\qquad\textrm{for all }v \in d\left(\Gamma'\right)_p.
\end{equation}
To prove $x$ exists, argue using Hensel's lemma and the Chinese remainder theorem, as in Proposition~\ref{prp: cyclic case}. Using $\{v_2,v_3\}$ and $\{2dv_4,v_5\}$ as bases for $d(\Gamma)_p$ and $d(\Gamma')_p$, one can check that the transformation
\[
\phi = \begin{pmatrix}
1 & 0 \\
-\frac{d(1 + 4c_{\alpha'}d)}{p} & -2d
\end{pmatrix}
\]
is a witness to~\eqref{eq: pSylow}, where $c_{\alpha'} \in \{0,\dots,p-1\}$ is the constant for the lattice $\Gamma'$ such that $p\mid(1 + 4c_{\alpha'}d)$.
\end{proof}

Finally, we treat the case when $p \mid d$.  It cannot also be the case that $p \mid (1 + 4c_\alpha d)$, so the lattices in Theorem~\ref{thm: lattices} with $\ia = 1$ and $p \mid (1 + 4c_\alpha d)$ cannot occur.  This leaves three possible distinct discriminant groups for lattices $\Ga$.  First, isomorphism classes of lattices with cyclic discriminant group are handled in Proposition~\ref{prp: cyclic case}. Second, there is only one lattice, up to isomorphism, with discriminant group ${\Z}/2d{\Z} \oplus {\Z}/p{\Z} \oplus {\Z}/p{\Z}$, characterized by $\ia = 0$ and $c_\alpha = 0$.  Thus, it remains to understand when two lattices with discriminant group ${\Z}/2d{\Z} \oplus {\Z}/p^2{\Z}$ are isomorphic to each other.

\begin{lem}
\label{lem: p^2 divides d}
Let $p$ be an odd prime such that $p \mid d$. Then there  are two lattices $\Ga$, up to isomorphism, with discriminant group $\disc \cong {\Z}/2d{\Z} \oplus {\Z}/p^2{\Z}$.
\end{lem}

\begin{proof}
We show that two lattices $\Ga$ and $\Gamma_{\langle\alpha'\rangle}$ with discriminant group ${\Z}/2d{\Z} \oplus {\Z}/p^2{\Z}$ are isomorphic if and only if $c_\alpha/c_{\alpha'}$ is a quadratic residue modulo $p$. Write $d = p^e\cdot d_0$, where $p \nmid d_0$, so that 
\[
{\Z}/2d{\Z} \oplus {\Z}/p^2{\Z} \cong {\Z}/2d_0{\Z}\oplus {\Z}/p^e{\Z} \oplus {\Z}/p^2{\Z}.
\]
Using the notation of the proof of Theorem~\ref{thm: lattices}, we may assume that
\[
d(\Ga) = \langle p^ev_1\rangle \oplus \underbrace{\langle 2d_0v_1, v_2 \rangle}_{d(\Ga)_p}\qquad\textrm{and}\qquad
d(\Gamma_{\langle\alpha'\rangle}) = \langle p^ev_1\rangle \oplus \underbrace{\langle 2d_0v_1, v_2'\rangle}_{d(\Gamma_{\langle\alpha'\rangle})_p},
\]
where $2dp^2v_2' = (0,2dp,-4c_{\alpha'} d)$. Recall that $p^ev_1$ and $2d_0v_1$ have orders $2d_0$ and $p^e$, respectively, while $v_2$ and $v_2'$, each have order $p^2$. 

As in the proof of Lemma~\ref{lem: p not div d}, the quadratic forms $q_\alpha$ and $q_{\alpha'}$ associated to our lattices are isomorphic if and only if there is an $\phi \in \Aut({\Z}/p^e{\Z}\oplus {\Z}/p^2{\Z})$ such that
\[
(q_\alpha)_p\left(\phi(v)\right) \equiv (q_{\alpha'})_p(v) \pmod{2{\Z}}\qquad \textrm{for all }v\in d(\Gamma_{\langle\alpha'\rangle})_p
\]
The symmetric matrices associated to $(q_\alpha)_p$ and $(q_{\alpha'})_p$ are respectively equal to
\[
\begin{pmatrix}
-\frac{2d_0}{p^e} & 0 \\
0 & -\frac{2c_\alpha}{p^2}
\end{pmatrix}
\qquad\textrm{and}\qquad
\begin{pmatrix}
-\frac{2d_0}{p^e} & 0 \\
0 & -\frac{2c_{\alpha'}}{p^2}
\end{pmatrix}.
\]

First, suppose that $c_\alpha/c_{\alpha'}$ is a quadratic residue modulo $p$. Then we can take
\[
\phi = 
\begin{pmatrix}
1 & 0 \\
0 & y
\end{pmatrix},
\]
where $y \in {\Z}/p^2{\Z}$ satisfies
\[
-\frac{2c_\alpha}{p^2} \equiv -\frac{2c_{\alpha'}y^2}{p^2} \pmod{2{\Z}}.
\]
Such a $y$ exists by Hensel's lemma and because $c_\alpha/c_{\alpha'}$ is a quadratic residue modulo $p$.

Now suppose that $(q_\alpha)_p$ and $(q_{\alpha'})_p$ are isometric. Then their associated bilinear forms must be isomorphic, and there is an $A \in \Aut({\Z}/p^e{\Z}\oplus {\Z}/p^2{\Z})$, which can be represented by a $2\times 2$ matrix, such that
\begin{equation}
\label{eq:matrixcong}
A^t \cdot 
\begin{pmatrix}
\frac{-2d_0}{p^e} & 0 \\
0 & \frac{-2c_\alpha}{p^2}
\end{pmatrix}
\cdot A \equiv 
\begin{pmatrix}
\frac{-2d_0}{p^e} & 0 \\
0 & \frac{-2c_{\alpha'}}{p^2}
\end{pmatrix}
\pmod{{\Z}}
\end{equation}
If $e \geq 3$ then we may assume that $A$ has the form $\displaystyle \begin{pmatrix}
a & 0 \\
b & c
\end{pmatrix}
$, in which case the $(2,2)$ entry in the congruence~\eqref{eq:matrixcong} reads
\begin{equation}
\label{eq: it's a QR!}
-\frac{2c_\alpha}{p^2} \equiv -\frac{2c_{\alpha'}c^2}{p^2} \pmod{{\Z}},
\end{equation}
and we conclude that $c_\alpha/c_{\alpha'}$ is a quadratic residue modulo $p$. If $e = 2$, then $A \in \GL_2({\Z}/p^2{\Z})$. Multiplying~\eqref{eq:matrixcong} by $p^2$ and taking determinants we arrive at the same conclusion. Finally, if $e = 1$, then we may assume that $A$ has the form $\displaystyle \begin{pmatrix}
a & b \\
0 & c
\end{pmatrix}
$, in which case the $(2,2)$ entry in the congruence~\eqref{eq:matrixcong} is again given by~\eqref{eq: it's a QR!}, and $c_\alpha/c_{\alpha'}$ is a quadratic residue modulo $p$.
\end{proof}

\subsection{Counting lattices}

We continue using the notation of Theorem~\ref{thm: lattices}; in particular, $S$ denotes a complex projective K3 surface with $\NS(S) \cong \Z h$. The purpose of this section is to count, for each nontrivial $\langle\alpha\rangle \subset \Br(S)[p]$, the number of lattices in each isomorphism class of $\Ga \subset T_S$.

Since $\Ga \subseteq T_S$ has index $p$, we know that
\[
pT_S \subseteq \Ga \subseteq T_S
\]
and thus
\[
H_\alpha := \Ga/pT_S \subseteq T_S/pT_S \cong \F_p^{21}.
\]
We may consider $H_\alpha$ as a hyperplane in $\F_p^{21}$. Conversely, to every hyperplane in $T_S/pT_S$, we may associate an index $p$ sublattice $\Ga$ of $T_S$.  Thus, the projective space $\P\left((T_S/pT_S)^*\right)$ parametrizes index $p$-sublattices of $T_S$. Using the identification $T_S \cong \Z v \oplus \Lambda'$, and setting $v^* = -v/2d$, the intersection form on $T_S$ allows us to identify 
\[
\P\left((T_S/pT_S)^*\right) \cong \P\left(\F_p\cdot v^* \oplus \Lambda'/p\Lambda'\right).
\]
The set of index $p$ lattices $\Ga$ that have $\ia = 0$ are then identified with
\[
\P^{19}(\F_p) = \P\left(\Lambda'/p\Lambda'\right) \subseteq \P\left(\F_p\cdot v^*\oplus \Lambda'/p\Lambda'\right)
\]
while the set of lattices with $\ia = 1$ can be identified with the distinguished open affine
\begin{align*}
\mathbb{A}^{20}(\F_p) = \Lambda'/p\Lambda' &\subseteq \P\left(\F_p\cdot v^*\oplus \Lambda'/p\Lambda'\right) \\
\lambda_\alpha &\mapsto [v^* + \lambda_\alpha]
\end{align*}
Define the quadratic form 
\begin{align*}
\mathfrak{Q} \colon \Lambda'/p\Lambda' &\to \F_p \\
\lambda &\mapsto -\frac{\langle\lambda,\lambda\rangle}{2} \bmod p.
\end{align*}
A lattice $\Ga$ is determined by the quantities $\ia$ and $c_\alpha$.  Recall that $-2c_\alpha = \langle \lambda_\alpha,\lambda_\alpha\rangle$, so for example, lattices with $\ia = 1$ and a prescribed  $c_\alpha$ correspond to $\F_p$-points on the affine quadric $\{ \mathfrak{Q}(x) = c_\alpha \}\subset \mathbb{A}^{20}$.  The following well-known lemma will help us count the required points on quadrics.  We include a proof here for completeness. 

\begin{lem}
\label{lem: counts}
Let $p$ be an odd prime, and let $Q$ be a nondegenerate, homogeneous quadratic form in $2n$ variables over $\F_p$.  Assume that $(-1)^n\textup{disc}(Q) \in \F_p^{\times 2}$. Write $f(n)$ for the number of zeroes of $Q$ (including the trivial zero). For $0\neq i \in \F_p$, let $g_i(n)$ denote the number of solutions to the equation $Q = i$. Then 
\[
f(n) = p^{n-1}(p^n+p-1)\qquad\textrm{and}\qquad
g_i(n) = p^{n-1}(p^n-1).
\]
In particular, $g(n) := g_i(n)$ is independent of $i$. 
\end{lem}

\begin{proof}
The hypothesis on $p$ and $Q$ imply that $Q$ is isometric to the form $Q\cong x_1x_2+\cdots+x_{2n-1}x_{2n}$. We then note that $f(n)$ satisfies the recurrence relation
\[
f(n)=f(n-1)(2p-1)+(p^{2n-2}-f(n-1))(p-1),
\]
because, informally,
\[
\begin{split}
f(n)&=\#(\textrm{zeroes of }x_1x_2 + \cdots + x_{2n-3}x_{2n-2})\cdot \#(\textrm{zeroes of }x_{2n-1}x_{2n}) \\
&\quad +\#(\textrm{nonzero values of }x_1x_2 + \cdots x_{2n-3}x_{2n-2})\cdot\#(\textrm{zeroes of }x_{2n-1}x_{2n}-i)),
\end{split}
\]
where $i\in \mathbb F_p^\times$.  The initial condition $f(1) = 2p - 1$ then allows us to determine $f(n)$, and we obtain the claimed quantity.  The function $g_i(n)$ obeys the same recurrence relation, but with initial condition $g_i(1) = p - 1$.
\end{proof}

We begin by counting lattices in isomorphism classes with cyclic discriminant group.  The following proposition is a complement to Proposition~\ref{prp: cyclic case}.

\begin{prp}
\label{prp: cyclic case bis}
Let $p$ be an odd prime.  Consider the lattices $\Ga$ in Theorem~\ref{thm: lattices} for which $\disc \cong {\Z}/2dp^2{\Z}$ is a cyclic group, generated by an element $v$. Write $q\colon \disc \to \Q/2{\Z}$ for the discriminant form of $\Ga$.
\begin{enumerate}
\item[(i)] If $p\mid d$ then there are $p^{20}$ such lattices, all isomorphic to each other.
\item[(ii)] If $p\nmid d$, then there are two isomorphism classes of lattices.  The isomorphism class corresponding to the case where $-2dp^2q(v)$ is a square modulo $p$ has $\frac{1}{2} p^{10} \left(p^{10}+1\right)$ lattices.  The other class has $\frac{1}{2} p^{10} \left(p^{10}-1\right)$ lattices.
\end{enumerate}
\end{prp}

\begin{proof} We have discussed the isomorphism classes of $\Ga$ in Proposition~\ref{prp: cyclic case}, so we focus on the lattice counts. If $p \mid d$  then $p\nmid (1+4\ca d)$. Hence $\disc$ is cyclic if and only if $\ia = 1$.  Lattices with $\ia = 1$ are in one-to-one correspondence with points in the distinguished open affine $\mathbb{A}^{20}(\F_p) = \Lambda'/p\Lambda' \subseteq \P\left(\F_p\cdot v^*\oplus \Lambda'/p\Lambda'\right)$. There are thus $p^{20}$ such lattices.

Next, suppose that $p\nmid d$.  Let us call the two isomorphism classes in part (ii) of the proposition $C_s$ and $C_n$\footnote{Here $C$ stands for cyclic, $s$ stands for $-2dp^2q(v)$ is a square modulo $p$ and $n$ stand for $-2dp^2q(v)$ is a non-square modulo $p$.}. Let $C_{s,\ia}=\{\Ga\in C_s \mid \Ga = \Gamma_{\ia,\ca} \textrm{ for some }\ca\}$ and similarly for $C_{n,\ia}$.  If $\ia=0$ then  $\disc$ is cyclic only if $p\nmid \ca$.  In this case $\disc$ is generated by $v_1+v_2$ and $-2dp^2q(v_1+v_2)=p^2+4\ca d$.  This is a square modulo $p$ if and only if $\ca d$ is a square modulo $p$.  Note that
\[
\#\{x \in \F_p^\times\mid x d \in \F_p^{\times 2}\}=(p-1)/2 \quad\textrm{and}\quad
\#\{x \in \F_p^\times\mid x d \notin \F_p^{\times 2}\}=(p-1)/2.
\]  
In particular, of the $p-1$ non-zero $\ca$'s mod $p$, there are $\displaystyle (p-1)/2$ such that  $\ca d$ is a square $\pmod p$. Therefore, using the notation of Lemma~\ref{lem: counts}, we have
\[
\#\{\lambda \in (\mathbb{A}^{20}\setminus \{0\})(\F_p) \mid \mathfrak{Q}(\lambda) \in \F_p^{\times 2}\}=\frac{p-1}{2}\cdot g(10).  
\]
Since $\ia =0$, the $\lambda$ are in $\mathbb P^{19}(\mathbb F_p)$ and  we must divide our count by $p-1$ to obtain
\[
\#C_{s,0} = \frac{p-1}{2}\cdot \frac{g(10)}{p-1}=\frac{p^{9}(p^{10}-1)}{2}.
\]
The same calculation shows $\displaystyle \#C_{n,0}=\frac{p^{9}(p^{10}-1)}{2}$.

If $\ia=1$ then $\disc$ is cyclic only if $p\nmid(1+4\ca d)$.   In this case $\disc$ is generated by $v_4$ and $-2dp^2q(v_4)=1+4\ca d$.  Since $1+4\ca d \equiv 1+4c_{\alpha'} d \pmod p$ if and only if $\ca \equiv c_{\alpha'} \pmod p$, we see that as sets $\F_p=\{1+4\ca d|\ca \in \F_p\}$.  Therefore
\[
\#\{\ca \mid 0\ne1+4\ca d\ne x^2 \textrm{ for all }x \in \F_p^\times\} = \frac{p-1}{2}.
\]
Since $\ca=0$ makes $1+4\ca d$ a square modulo $p$, we see that 
\begin{align*}
\#C_{s,1}&=f(10)+\left(\frac{p-1}{2}-1\right)g(10) =\frac{1}{2}p^9(p^{11}-p^{10}+p+1),\\
\#C_{n,1}&=\frac{p-1}{2} g(10)=\frac{1}{2}p^{9}(p^{10}-1)(p-1).
\end{align*}
Finally,
\begin{align*}
\#C_{s}&=\#C_{s,0}+\#C_{s,1}=\frac{1}{2} p^{10} \left(p^{10}+1\right),\\
\#C_{n}&=\#C_{n,0}+\#C_{n,1}=\frac{1}{2} p^{10} \left(p^{10}-1\right).
\end{align*}
\end{proof}

\begin{prp}
Suppose that $p$ is an odd prime with $p \nmid d$. There are $\displaystyle\frac{p^{20} - 1}{p - 1}$ lattices $\Ga$ in Theorem~\ref{thm: lattices} with noncyclic discriminant group.
\end{prp}

\begin{proof}
This is clear, since there are a total of $\frac{p^{21} - 1}{p - 1}$ lattices $\Ga$ and the ones with cyclic group account for a total of $p^{20}$ lattices.  
\end{proof}

\begin{prp}
Suppose that $p$ is an odd prime with $p\mid d$. Consider lattices $\Ga$ as in Theorem~\ref{thm: lattices}. There are
\begin{itemize}
\item $\frac{1}{2}p^9(p^{10} - 1)$ lattices with $\ia = 0$, $p \nmid c_\alpha$ and $c_\alpha$ is a quadratic residue modulo $p$,
\item $\frac{1}{2}p^9(p^{10} - 1)$ lattices with $\ia = 0$, $p \nmid c_\alpha$ and $c_\alpha$ is a quadratic nonresidue modulo $p$,
\item $\frac{(p^9 + 1)(p^{10} - 1)}{p - 1}$ lattices with $\ia = 0$ and $p \mid c_\alpha$,
\item $p^{20}$ lattices with $\ia = 1$.
\end{itemize}
\end{prp}

\begin{proof}
The first two types of lattices can be counted the same way we counted $C_{s,0}$ and $C_{n,0}$ in Proposition~\ref{prp: cyclic case bis}. The third type corresponds to $\F_p$-points of a smooth quadric in $\P^{19}$, of which there are $\frac{f(10) - 1}{p - 1}$.  The fourth type were counted in Proposition~\ref{prp: cyclic case bis}.
\end{proof}

\newpage

\subsection{Summary}
\label{ss: summary}

\begin{thm}
\label{thm: lattices clean}
Let $S$ be a complex projective K3 surface of degree $2d$ with N\'eron-Severi group $\NS(S)$ isomorphic to $\Z h$, and write $T_S := \langle h\rangle^\perp \subset {\rm H}^2(S,\Z)$ for its transcendental lattice. Let $p$ be an odd prime, and let $\alpha \in \Br(S)[p]$ be nontrivial, with associated index $p$ sublattice $\Ga=\Gamma_{\ia,\ca} \subset T_S$.  Write $q\colon \disc \to \Q/2{\Z}$ for the discriminant form of $\Ga$.
\begin{enumerate}
\item If $p \nmid d$ then there are three isomorphism classes of lattices $\Gamma_{\alpha}$.  They are classified in Table~\ref{ta: summary1}.
\begin{table}[ht]\renewcommand{\arraystretch}{2}
\[
\begin{tabular}{c|c|c}
$\disc$ & Distinguishing Feature & Number of $\Ga$ \\
\hline
${\Z}/2dp^2{\Z} = \langle v\rangle$ & $-2dp^2q(v) \equiv \Box \bmod p$ & $\displaystyle\frac{1}{2}p^{10}(p^{10} + 1)$\\[1ex]
\hline
${\Z}/2dp^2{\Z} = \langle v\rangle$ & $-2dp^2q(v) \not\equiv \Box \bmod p$ & $\displaystyle\frac{1}{2}p^{10}(p^{10} - 1)$\\[1ex]
\hline
${\Z}/2d{\Z} \oplus {\Z}/p{\Z}\oplus {\Z}/p{\Z}$ &  & $\displaystyle\frac{p^{20} - 1}{p - 1}$\\
\end{tabular}
\]
\caption{Sublattices $\Ga = \Gamma_{\ia,c_\alpha} \subseteq T_S$ of index $p\nmid d$. The symbol $\Box$ stands for ``a square''.}
\label{ta: summary1}
\end{table}

\item If $p \mid d$ then there are four isomorphism classes of lattices $\Gamma_{\alpha}$. They are classified in Table~\ref{ta: summary2}.

\begin{table}[ht]\renewcommand{\arraystretch}{2}
\[
\begin{tabular}{c|c|c|c}
$\disc$ & $\ia$ & $\ca\bmod p$ & Number of $\Ga$ \\
\hline
${\Z} /2d\Z\oplus \Z /p^2\Z$ & $0$ & quadratic residue & $\displaystyle\frac{1}{2}p^{9}(p^{10} - 1)$\\[1ex]
\hline
${\Z} /2d\Z\oplus \Z /p^2\Z$ & $0$ & quadratic nonresidue & $\displaystyle\frac{1}{2}p^{9}(p^{10} - 1)$\\[1ex]
\hline
${\Z} /2d\Z\oplus \Z /p\Z\oplus \Z /p\Z$ & $0$ & $0$  & $\displaystyle\frac{(p^9 + 1)(p^{10} - 1)}{p-1}$\\[1ex]
\hline
${\Z}/2dp^2{\Z}$ & $1$ & no restriction & $p^{20}$\\
\end{tabular}
\]
\caption{Sublattices $\Ga = \Gamma_{\ia,c_\alpha} \subseteq T_S$ of index $p\mid d$. }
\label{ta: summary2}
\end{table}
\end{enumerate}
\end{thm}

\newpage

\subsection{Lattice theory for Mukai dual K3 surfaces}
\label{ss: mukai duals}

Let $S$ be a complex K3 surface of degree two with $\NS(S) \cong \Z h$.  In \S9 of \cite{vangeemen05}, van Geemen showed that sublattices of index two $\Gamma_{\langle\alpha\rangle}$ of $T_S$ in a particular isomorphism class naturally give rise to K3 surfaces of degree eight via a primitive embedding $\Gamma_{\langle\alpha\rangle} \hookrightarrow \Lambda_{\mathrm K3}$, using the surjectivity of the period map for K3 surfaces.  In this section, we explain a well-known generalization of this framework due to Mukai~\cite{mukai87}.

Using the notation of Section~\ref{ss: set-up}, we let $S$ denote a complex projective K3 surface of degree $2d$ with $\NS(S) \cong \Z h$, and we fix a primitive vector $v \in \H^2(S,\Z)$ such that $v^2 = -2d$ and $T_S \cong \langle-2d\rangle \oplus \Lambda'$. Then $\Gamma := \langle-2dp^2\rangle \oplus \Lambda'$ is a sublattice of index $p$ of $T_S$.  One checks that $d\left(\Gamma\right) \cong \Z/2dp^2\Z$, generated by $\frac{v}{2dp}$, from which it follows that $-2dp^2q(v) = 1$, which is a square modulo $p$.  By Theorem~\ref{thm: lattices clean}, if $p \nmid d$, then there are $\frac{1}{2}p^{10}(p^{10} + 1)$ sublattices of $T_S$ of index $p$ isomorphic to $\Gamma$, and if $p \mid d$, there are $p^{20}$ such lattices.  Up to isomorphism there is a unique primitive embedding $\Gamma \hookrightarrow \Lambda_{\textrm{K3}}$. By the surjectivity of the period map for K3 surfaces, there is a complex projective K3 surface $Y$ of degree $2dp^2$ such that $T_Y \cong \Gamma$.  Moreover, we obtain an isomorphism of rational Hodge structures $\H^2(Y,\Q) \to \H^2(S,\Q)$.

The Hodge isogenies above appear naturally in Mukai's work on moduli spaces of stable sheaves on K3 surfaces (see Section~\ref{ss: mukai set-up} below for precise definitions).  In this context, one starts with a polarized K3 surface $X$ of degree $2dp^2$ such that $\NS(X) = \Z h'$, and one defines $S := M_X(p,h',dp)$ to be the moduli space of sheaves on $X$ with Mukai vector $v' := (p,h',dp)$.  Then $S$ is a K3 surface, and by Theorem~1.5 of~\cite{mukai87}, there is an isomorphism of Hodge structures between $v'^\perp/\Z v'$ and $\H^2(S,\Z)$, which is compatible with the Mukai pairing on $v'^\perp/\Z v'$ and cup product on $\H^2(S,\Z)$. Moreover, there is a rational Hodge isometry
\[
\H^2(X,\Q) \to v'^\perp/\Z v' \otimes \Q \qquad x \mapsto \left(0,x,\frac{h'\cdot x}{p}\right)
\]
mapping $h'/p$ to $h := (0,h'/p,2d)$, which is integral in $v'^\perp/\Z v' \otimes \Q$ because $(0,h'/p,2d) - v'/p = (-1,0,d)$.  Composing the two isometries, we obtain a rational Hodge isometry 
\[
\phi\colon\H^2(X,\Q) \to \H^2(S,\Q)
\]
mapping $h'/p$ to an integral class $h \in H^2(S,\Q)$ such that $h^2 = 2d$, giving $S$ a polarization of degree $2d$.  The isometry $\phi$ induces an injection $T_X \hookrightarrow T_S$ whose image has index $p$.  As in Section~\ref{ss: set-up}, since $\NS(X) \cong \Z h'$, we have $T_X \cong \langle -2dp^2\rangle \oplus \Lambda'$, so $d\left(T_X\right) \cong \Z/2dp^2\Z$, and $-2dp^2q(u)$ is a square modulo $p$, where $q$ is the discriminant form on $d\left(T_X\right) = \langle u\rangle$.  Mukai's moduli spaces of stable sheaves therefore give geometric manifestations to the lattice theory discussed in this section.

In Section~\ref{s: Mukai}, we explain some geometric constructions realizing the Mukai duality between the surfaces $X$ and $S$ via projective dualities, and in Section~\ref{S: Failure of WA}, we use one of these constructions (the case $d = 1$ and $p = 2$) for an arithmetic application.

\subsection{Special Cubic fourfolds}
\label{ss: special 4-folds}

Continuing the theme of the previous section, we explore the connection between certain sublattices of $T_S$ of index $p$ on general K3 surfaces and special cubic fourfolds.   Geometric correspondences explaining these lattice-theoretic connections have arithmetic applications: such correspondences can yield Brauer-Severi bundles representing a generator for an order $p$ subgroup of $\Br(S)$.  This idea was exploited in~\cite{hvv11} to obtain counterexamples to weak approximation on a K3 surface, starting from a cubic fourfold containing a plane.  The results of this section are easily derived from general work of Hassett~\cite{h00} on special cubic fourfolds.  We include them here to alert the arithmetically inclined audience about a source of constructions of transcendental Brauer classes on K3 surfaces.  For example, the results of this section suggest that cubic fourfolds containing a del Pezzo surface of degree $6$ form a source of transcendental $3$-torsion elements on a K3 surface of degree $2$.  It would be very interesting to have a geometric correspondence capable of producing such a $3$-torsion element, as a Brauer-Severi bundle, starting from the special cubic fourfold.

Recall that a special cubic fourfold $Y \subseteq \P^5$ is a smooth cubic fourfold that contains a surface $T$ not homologous to a complete intersection. Let $h$ denote the hyperplane class of $\P^5$; assume that the lattice $K := \langle h^2, T\rangle \subset H^4(Y,\Z)$ is saturated.  The discriminant of $(Y,K)$ is the determinant of the Gram matrix of $K$.  The nonspecial cohomology of $(Y,K)$ is the orthogonal complement $K^\perp$ of $K$ with respect to the intersection form.  By Theorem~1.0.1 of~\cite{h00}, special cubic fourfolds $(Y,K)$ of discriminant $D$ form an irreducible divisor $\calC_D$ of the moduli space of cubic fourfolds.  In what follows, we write $K_{D}$ for the special cohomology lattice of a special cubic fourfold in $\calC_{D}$.

\begin{prp}
\label{prop:cubicfourfolds}
Let $S$ be a $K3$ surface of degree $2d$ with $\mathrm{NS}(S) = \Z h$ and $p$ a prime number.  For a nonzero $\alpha \in \Br(S)[p]$, denote by $\Ga$ its corresponding sublattice of index $p$ in $T_S$.
\begin{enumerate}
\item Suppose that $d = 1$ and $p > 3$. There is precisely one isomorphism class of lattices $\Ga \subset T_S$ such that there exist a special cubic fourfold $(Y,K)$ of discriminant $2p^2$, and an isomorphism of lattices $\Ga \cong (K^\perp)(-1)$.
\item Suppose that $p = 3$. There is an isomorphism class of lattices $\Ga \subset T_S$ such that there exist a special cubic fourfold $(Y,K)$ of discriminant $18d$, and an isomorphism of lattices $\Ga \cong (K^\perp)(-1)$ if and only if $(6,d) = 1$ and if $q$ is a prime dividing $d$, then $q \equiv 1 \bmod 3$.
\end{enumerate}
\end{prp}

\begin{rmk}
The divisor $\calC_D$ is nonempty if and only if $D > 6$ and $D \equiv 0$ or $2 \bmod 6$ (see Theorem~1.0.1 of~\cite{h00}). This implies that if $p > 3$ and $d \equiv 2 \bmod 6$, then no special cubic fourfolds have nonspecial cohomology isomorphic to a twist of an index $p$ sublattice $\Ga \subset T_S$ for a K3 surface $S$ of degree $2d$.
\end{rmk}

\begin{proof}[Proof of Proposition~\ref{prop:cubicfourfolds}]\ 

\noindent (1.) In Proposition 3.2.5 of~\cite{h00}, Hassett shows that the discriminant group $d(K_{2p^2}^\perp(-1))$ is cyclic, and a generator $u$ can be chosen so that its value for the discriminant quadratic form is $-\frac{4p^2 - 1}{6p^2}$. Then
\[
-2p^2q(u) = -2p^2 \cdot \left(-\frac{4p^2 - 1}{6p^2}\right) \equiv -\frac{1}{3} \bmod p,
\]
so $-2p^2q(u)$ is a quadratic residue modulo $p$ if $p \equiv 1 \bmod 3$ and otherwise it is a quadratic nonresidue. Let $\Ga$ be an index $p$ sublattice of $T_S$ with cyclic discriminant group $d(\Ga)$ generated by $v$, and such that the quadratic characters of $-2p^2q(v)$ and  $-2p^2q(u)$ coincide (such a lattice exists, and is unique up to isomorphism, by Theorem~\ref{thm: lattices clean}). We claim that $\Ga \cong K_{2p^2}^\perp(-1)$. Since $d = 1$, using Theorem~\ref{thm: lattices}, either $i_\alpha = 0$ and $p \nmid \ca$, or $i_\alpha = 1$ and $p \nmid 1 + 4c_\alpha$.  On the other hand, Proposition~\ref{prp: cyclic case} shows that the isomorphism class of $\Ga$ depends only on whether $-2p^2q(v)$ is a square modulo $p$ or not.  Thus, each lattice $\Ga$ with $i_\alpha = 1$ and $p \nmid 1 + 4c_\alpha$ is isomorphic to a lattice $\Gamma_{\langle\alpha'\rangle}$ with $i_{\alpha'} = 0$ and $p \nmid c_{\alpha'}$. Consequently, we may assume that $i_\alpha = 0$ and $p \nmid \ca$. As in the proof of Proposition~\ref{prp: cyclic case}, there is a generator whose value for the discriminant form is $-\frac{1}{2p^2}(p^2 + 4\ca)$. Therefore, there is an isomorphism $\Ga \cong K_{2p^2}^\perp(-1)$if and only if there exists an $x \in (\Z/2p^2\Z)^\times$ such that
\[
-\frac{4p^2 - 1}{6p^2} \equiv -\frac{p^2 + 4\ca}{2p^2}x^2 \bmod 2\Z.
\]
Multiplying by $2p^2$ this becomes
\[
\frac{4p^2 - 1}{3} \equiv (p^2 + 4\ca)x^2 \bmod 4p^2\Z.
\]
Modulo $4$ there is always such an $x$.  Modulo $p$, we need $1 \equiv -3\ca x^2 \bmod p$.
If $p \equiv 1 \bmod 3$, then $-3$ and $\ca$ are squares modulo $p$.
If $p \equiv 2 \bmod 3$, then both are not squares modulo $p$.
Thus $-3\ca$ is always a square modulo $p$.

\noindent (2.) We know from Proposition~3.2.5 in~\cite{h00} that $d(K_{18d}^\perp(-1)) \cong \Z/6d\Z \oplus \Z/3\Z$. If $3$ divides $d$, then no discriminant group in Theorem~\ref{thm: lattices clean} is isomorphic to $d(K_{18d}^\perp(-1))$. Assume that $3 \nmid d$; then $d(K_{18d}^\perp(-1)) \cong \Z/2d\Z \oplus \Z/3\Z \oplus \Z/3\Z$. Let $u = (3,0) \in \Z/6d\Z \oplus \Z/3\Z$ be a generator for the subgroup $\Z/2d\Z$ of $d(K_{18d}^\perp(-1))$. Using Proposition 3.2.5 of~\cite{h00}, we have $q(u)= \frac{3}{2d}$. By Theorem~\ref{thm: lattices clean}, there is unique isomorphism class $\Ga$ with $d(\Ga) = \Z/2d\Z \oplus \Z/3\Z \oplus \Z/3\Z$; without loss of generality, we may assume that $i_\alpha = 0$ and $\ca = 0$. In this case, the vector $v_1$ in the proof of Theorem~\ref{thm: lattices} is a generator for the subgroup $\Z/2d\Z$ of $d(\Ga)$, and its value for the quadratic form is $-\frac{1}{2d}$. Thus, to have an isomorphism $d(\Ga) \cong d(K_{18d}^\perp(-1))$, we need $x \in (\Z/2d\Z)^\times$ such that
\[
\frac{3}{2d} \equiv -\frac{1}{2d}x^2 \bmod 2\Z.
\]
Multiplying by $2d$ we have
\[
3 \equiv -x^2 \bmod 4d.
\]
Such an $x$ exists if and only if $2 \nmid d$ and if for any prime $q \mid d$, we have $q \equiv 1 \bmod 3$. This shows that the conditions on $d$ in the statement of the proposition are necessary.  To see they are sufficient, we need only show that the $3$-Sylow part of $d(K_{18d}^\perp(-1))$ is isometric to the $3$-Sylow part of $d(\Ga)$. The intersection forms of $d(K_{18d}^\perp(-1))_3$ and $d(\Ga)_3$ are given by
\[
\begin{pmatrix}
\frac{2}{3}d & 0 \\ 0 & -\frac{2}{3}
\end{pmatrix}
\quad\textup{and}\quad
\begin{pmatrix}
0 & \frac{2}{3} \\ \frac{2}{3} & 0
\end{pmatrix},
\]
respectively. Note that under the necessary conditions, we have $d \equiv 1 \bmod 6$. This implies that $\frac{2}{3}d \equiv \frac{2}{3} \bmod 2 \Z$. It follows from this that the two discriminant forms are isometric.
\end{proof}

\begin{rmk}
The proof of Proposition~\ref{prop:cubicfourfolds} shows that, when $d=1$ and $p \equiv 1 \bmod 3$, the twisted nonspecial cohomology of a special cubic fourfold of discriminant $2p^2$ and the transcendental lattice of a general K3 surface of degree $2p^2$ are \emph{both} isomorphic to the same sublattice of index $p$ in $T_S$.  While seemingly surprising, this phenomenon reflects the existence of associated K3 surfaces, in the sense of Hassett, for cubic fourfolds in $\calC_{2p^2}$; see Theorem~5.1.3 of~\cite{h00}.
\end{rmk}

Next, we elaborate on the geometric connection between special cubic fourfolds and K3 surfaces when $d = 1$ and $p > 3$. Let $\mathcal D'$ be the local period domain of marked special cubic fourfolds $(Y, K_{2p^2})$ of discriminant $2p^2$. The domain $\mathcal D'$ is an open subset of a quadratic hypersurface in $\mathbb P(K_{2p^2}^{\perp}\otimes \mathbb C)$ and is a connected component of 
\[
\{ [\omega] \in \mathbb P(K_{2p^2}^{\perp}\otimes \mathbb C) \mid (\omega, \omega) = 0, \quad (\omega, \bar{\omega}) < 0 \}.
\]
Let $\Gamma_{2p^2}$ be the arithmetic group such that $\mathcal D := \Gamma_{2p^2} \backslash \mathcal D'$ is the global period domain of special cubic fourfolds of discriminant $2p^2$ (see Section~2.2 of~\cite{h00} for the definition of this group). The Torelli theorem for cubic fourfolds implies that the period map $\mathcal C_{2p^2} \rightarrow \mathcal D$ is an open immersion~\cite{voisin}. Let $\Lambda_2$ be the primitive cohomology lattice of a degree two polarized K3 surface. Write $\mathcal N'$ for the local period domain of K3 surfaces of degree two, which is an open domain of
\[
\{ [\omega] \in \mathbb P(\Lambda_2\otimes \mathbb C) \mid (\omega, \omega) = 0, \quad (\omega, \bar{\omega}) > 0 \}.
\]
Let $\Gamma_2$ be the arithmetic group such that $\mathcal N := \Gamma_2\backslash \mathcal N'$ is the global period domain for K3 surfaces of degree two. Surjectivity of the period map for K3 surfaces identifies the global period domain $\mathcal N$ with the coarse moduli space $\mathcal K_2$ of degree two polarized K3 surfaces.
\begin{prp}
\label{prop:mapofmoduli}
An embedding $j\colon K^\perp_{2p^2} \hookrightarrow {-\Lambda}_2$ induces a dominant morphism $\mathcal C_{2p^2} \to \mathcal K_2$ of quasi-projective varieties. 
\end{prp}
\begin{proof}
The divisor $\calC_{2p^2}$ is algebraic by Theorem 3.1.2 of~\cite{h00}.  Next, we describe a holomorphic map $\mathcal C_{2p^2} \to \mathcal K_2$, by mirroring the argument of Lemma~3.2 of \cite{kondo}.  Proposition~\ref{prop:cubicfourfolds} allows us to identify the local period domain $\mathcal D'$ with $\mathcal N'$. We show that $O(K_{2p^2}^\perp) \subset O(\Lambda_2)$. Identifying $K_{2p^2}^\perp$ with a index $p$ sublattice of ${-\Lambda}_2$, we may consider the subgroup $M = ({-\Lambda}_2)/K_{2p^2}^\perp$ of $d(K_{2p^2}^\perp)$, which is isotropic. By Proposition~1.4.1 of \cite{nikulin}, we have
\[
\Lambda_2 = \{ x \in (K_{2p^2}^\perp)^* \mid x \bmod K_{2p^2}^\perp \in M \}.
\]
Any map $\varphi \in O(K_{2p^2}^\perp)$ naturally extends to $(K_{2p^2}^\perp)^*$. Hence $\varphi$ induces an isomorphism on $d(K_{2p^2}^\perp)$. The group $d(K_{2p^2}^\perp)$ being cyclic, $M$ is preserved by $\varphi$. This shows that $\varphi$ induces an isomorphism on $\Lambda_2$, i.e., $\varphi \in O(\Lambda_2)$.  To see that the holomorphic map $\mathcal C_{2p^2} \to \mathcal K_2$ obtained thus far is algebraic, one argues as in the proof of Proposition~2.2.2 of~\cite{h00}.  The morphism is dominant because the map $\calC_{2p^2} \to \calD$ is an open immersion.
\end{proof}

\begin{rmk}
There is an analogous morphism of coarse moduli spaces $\calK_{2p^2} \to \calK_2$ encoding the Mukai duality explained in Section~\ref{ss: mukai duals}. Kond\={o} has studied this morphism in detail~\cite{kondo}; it has degree $p^{10}(p^{10} + 1)$. 
\end{rmk}

\section{Mukai dual K3 surfaces}
\label{s: Mukai}

Moduli spaces of sheaves on K3 surfaces were first studied by Mukai~\cite{mukai84,mukai88i}. The theory was further developed by G{\"o}ttsche and Huybrechts~\cite{gh96}, O'Grady~\cite{ogrady97}, Yoshioka~\cite{yoshioka01}, and others. We will mainly be interested in two-dimensional moduli spaces; see Mukai~\cite{mukai87}. A general reference for these moduli spaces is the book of Huybrechts and Lehn~\cite{hl97}. The modern approach also relies heavily on Fourier-Mukai transforms~\cite{mukai81}, and their twisted version due to C{\u a}ld{\u a}raru~\cite{caldararu00,caldararu02}. A general reference for these derived equivalences is the book of Huybrechts~\cite{huybrechts06}.

Our goal in this section is to elaborate on Section~\ref{ss: mukai duals} and describe certain K3 surfaces and their Mukai dual surfaces, which are (twisted) derived equivalent. The Mukai dual surfaces can be described in several ways: they are moduli spaces of stable sheaves on the original K3 surface, and they also arise via projective duality. The former approach leads us to natural elements of the Brauer group, whereas the latter gives explicit equations for the Mukai dual surface, and is therefore indispensible for arithmetic applications.

\subsection{Set-up}
\label{ss: mukai set-up}

Let $X$ be a generic K3 surface of degree $2k$. We use generic to mean that $X$ belongs to a Zariski open subset of the moduli space of all K3 surfaces with primitive ample divisors $h$ with $h^2=2k$. The condition $\mathrm{NS}(X)\cong\mathbb{Z}h$ is sufficient to ensure genericity of $X$. Fix a \defi{Mukai vector}
$$v=(a,bh,c)\in\H^0(X,\Z)\oplus\H^2(X,\Z)\oplus\H^4(X,\Z),$$
and define $M_X(v)$ to be the moduli space of stable sheaves $\mathcal{E}$ on $X$ with Mukai vector
$$v(\mathcal{E}):=ch(\mathcal{E})\mathrm{Td}^{1/2}=\left(r,c_1(\mathcal{E}),r+\frac{1}{2}c_1(\mathcal{E})^2-c_2(\mathcal{E})\right)=v.$$
Here \defi{stable} means $\mu$-stable with respect to the polarization $h$ of $X$, i.e., any proper subsheaf $\mathcal{F}$ of $\mathcal{E}$ must have slope
$$\mu(\mathcal{F}):=\frac{c_1(\mathcal{F})\cdot h}{r(\mathcal{F})}$$
strictly less than the slope
$$\mu(\mathcal{E}):=\frac{c_1(\mathcal{E})\cdot h}{r(\mathcal{E})}$$
of $\mathcal{E}$. In the examples that interest us, $v$ will be primitive (i.e., $\mathrm{gcd}(a,b,c)=1$), in which case $\mu$-stability coincides with the related notion of Gieseker stability, and it also coincides with $\mu$-semistability and Gieseker semistability. 

Mukai~\cite{mukai84} proved that $M_X(v)$ is smooth of dimension
$$v^2+2:=b^2h^2-2ac+2,$$
and it admits a holomorphic symplectic structure. In particular, $v^2$ must be at least $-2$ if there exists a stable sheaf $\mathcal{E}$ with $v(\mathcal{E})=v$. When $v$ is primitive, $M_X(v)$ is compact; it is an irreducible symplectic variety. When $v$ is also isotropic (i.e., $v^2:=b^2h^2-2ac=0$), $S:=M_X(v)$ is a K3 surface~\cite{mukai87}. The degree of $S$ will be $2ac/\mathrm{gcd}(a,c)^2$. If $n:=\mathrm{gcd}(a,bh^2,c)=1$, then $S$ is a fine moduli space, the universal sheaf on $X\times S$ induces an equivalence between the derived categories of coherent sheaves on $X$ and $S$, and we say they are Mukai dual~\cite{mukai81}. If $n\neq 1$, then there is an $n$-torsion Brauer element $\alpha$ on $S$ obstructing the existence of a universal sheaf. Instead, there is a twisted universal sheaf and the derived category of $X$ is equivalent to the derived category of $\alpha$-twisted sheaves on $S$ (see C{\u a}ld{\u a}raru~\cite{caldararu00,caldararu02}).

Some particular cases are when $h^2=2k=2dn^2$ and $v=(n,h,nd)$. These cases were studied by Hassett and Tschinkel~\cite{ht00} ($d=1$), Iliev and Ranestad~\cite{ir07} ($d=2$ and $n=2$), the author~\cite{sawon07}, and Markushevich~\cite{markushevich06}. By the general principles outlined above, $S$ is a K3 surface of degree $2d$ that comes equipped with an $n$-torsion Brauer element $\alpha$.

In some low degree cases the Mukai duality can be realized by projective duality. The starting point is to describe $X$ as a linear section of a homogeneous variety. In the next three sections we will consider the cases
\begin{itemize}
\item $d=1$, $n=2$, producing a degree two K3 surface $S$ with a $2$-torsion Brauer element,
\item $d=1$, $n=3$, producing a degree two K3 surface $S$ with a $3$-torsion Brauer element,
\item $d=2$, $n=2$, producing a degree four K3 surface $S$ with a $2$-torsion Brauer element.
\end{itemize}
The first of these is precisely the remaining example of van Geemen.  In the first two cases, we represent the Brauer elements by Brauer-Severi varieties with fibres isomorphic to $\P^3$ and $\P^2$, respectively. 

For general K3 surfaces over $\C$, Huybrechts and Schr{\"o}er~\cite{hs03} proved that the Brauer group equals the cohomological Brauer group, i.e., the group of sheaves of Azumaya algebras up to equivalence is isomorphic to the torsion part of the analytic cohomology group $\H^2(S,\O^*)$. Their proof involves showing that any $n$-torsion element in $\H^2(S,\O^*)$ can be represented by a Brauer-Severi variety with fibres $\P^{n-1}$, which we call a {\em minimal\/} Brauer-Severi variety. For a K3 surface arising as a non-fine moduli space of sheaves, with an associated $n$-torsion Brauer element obstructing the existence of a universal sheaf, there are natural ways to represent the Brauer element as a Brauer-Severi variety, but in general they do not produce a minimal Brauer-Severi variety. In terms of sheaves of Azumaya algebras, the representatives are all Morita equivalent (naturally) but we do not necessarily obtain a sheaf of Azumaya algebras of minimal rank in this way. For example, in cases one and two above, the moduli space approach produces non-minimal Brauer-Severi varieties with fibres isomorphic to $\P^3$ and $\P^5$, respectively. However, there is also a geometric approach that produces a minimal Brauer-Severi variety in case two. We leave as an open question the existence of Brauer-Severi varieties in the third case.

\subsection{The degree eight/degree two duality}
\label{ss: 8to2}

A general degree eight K3 surface $X$ is a complete intersection of three quadrics in $\P^5$. To describe it as a linear section, we embed $Y:=\P^5=\P(V)$ in $\P^{20}=\P(\mathrm{Sym}^2V)$ using the Veronese embedding. The K3 surface $X$ will be the intersection of $Y\subset\P^{20}$ with a codimension three linear subspace $\P(U)=\P^{17}$.

Now we projectively dualize. The dual variety $\check{Y}$ is a sextic hypersurface in $\check{\P}^{20}=\P(\mathrm{Sym}^2V^*)$, the determinantal variety, and $\P(U^{\perp})=\P^2$ intersects this hypersurface in a plane sextic curve $C$. Here $U^{\perp}\subset\mathrm{Sym}^2V^*$ denotes the annihilator of $U$. We therefore obtain a degree two K3 surface $S$ as the double cover of $\P^2$ branched over the sextic $C$. We say that $X$ and $S$ are projectively dual varieties.

\begin{lem}
\label{lem: moduli}
The K3 surface $S$ can be naturally identified with the moduli space $M_X(2,h,2)$.
\end{lem}

\begin{prf}
This is Example~0.9 of Mukai~\cite{mukai84}, and Example~2.2 in~\cite{mukai88i}; a very detailed discussion is given by Ingalls and Khalid~\cite{ik13}. The basic idea is as follows. A point $p$ in $\P^2\subset\check{\P}^{20}$ corresponds to a hyperplane $H_p$ in $\P^{20}$, which intersects $Y=\P^5\subset\P^{20}$ in a quadric four-fold $Z_p$. The complete picture is:
$$\begin{array}{ccccccc}
 & & & \mbox{dually} & & & \\
Y={\P}^5={\P}(V) & \hookrightarrow & {\P}(\mathrm{Sym}^2V)={\P}^{20} & & \check{Y} & \hookrightarrow & {\P}(\mathrm{Sym}^2V^*)=\check{\P}^{20} \\
\cup & & \cup & & \cup & & \cup \\
Z_p & \hookrightarrow & H_p={\P}^{19} & & C & \hookrightarrow & {\P}(U^{\perp})={\P}^2 \\
\cup & & \cup & & & & \cup \\
X & \hookrightarrow & {\P}(U)={\P}^{17} & & & & p \\
\end{array}$$
If $p\in\P^2\backslash C$, then $Z_p$ is a smooth quadric. Now every smooth quadric in $\P^5$ can be identified with the Grassmannian $\mathrm{Gr}(2,4)$. The Grassmannian comes with two natural vector bundles of rank two: the universal bundle $E$ and the universal quotient bundle $F$, which fit in the exact sequence
$$0\rightarrow E\rightarrow \C^4\otimes\O\rightarrow F\rightarrow 0.$$
Dualizing gives
$$0\rightarrow F^*\rightarrow (\C^4)^*\otimes\O\rightarrow E^*\rightarrow 0,$$
so $E^*$ can also be regarded as a quotient bundle on $\mathrm{Gr}(2,4)$. Restricting $E^*$ and $F$ to $X$ via the embedding $X\subset Z_p\cong\mathrm{Gr}(2,4)$ yields two stable vector bundles on $X$ with Mukai vectors $v=(2,h,2)$ (see~\cite{ik13} for details, particularly page 450 and Corollary~3.5). Note that we had to choose an identification $Z_p\cong\mathrm{Gr}(2,4)$, but a different identification yields the same pair of bundles $E^*|_X$ and $F|_X$, up to interchanging them (the automorphism group of $\mathrm{Gr}(2,4)$ has two connected components, and as homogeneous bundles, $E^*$ and $F$ are invariant under pullbacks by automorphisms in the connected component of the identity, and interchanged by pullbacks by automorphisms in the other component). Alternatively, the identification $Z_p\cong\mathrm{Gr}(2,4)$ could be made canonical in the following way: Recall that there are two $\P^3$-families of maximal isotropic planes contained in a smooth quadric. The subfamily of maximal isotropic planes from, say, the first family that pass through a fixed point of the quadric will be parametrized by a line $\P^1$ in $\P^3$. Equivalently, each point of the quadric gives a plane $\C^2$ in $\C^4$, and this leads to an isomorphism of the quadric with $\mathrm{Gr}(2,4)$. The second family will yield a second isomorphism of the quadric with $\mathrm{Gr}(2,4)$, and of course the automorphism of $\mathrm{Gr}(2,4)$ given by composing these isomorphisms will interchange $E^*$ and $F$.

If $p\in C$, then $Z_p$ is a singular quadric. Assuming the K3 surface $X$ is generic, this singular quadric will always be of rank five, so $Z_p$ will be a cone over a smooth quadric threefold. This quadric in $\P^4$ can be identified with the lagrangian Grassmannian $\mathrm{LGr}(2,4)$. The blow-up $\tilde{Z}_p$ of $Z_p$ at the apex of the cone is therefore a $\P^1$-bundle over $\mathrm{LGr}(2,4)$. The embedding $X\subset Z_p$ lifts to an embedding $X\subset \tilde{Z}_p$ because $X$ does not contain the apex. Now the universal bundle $E$ and universal quotient bundle $F$ over $\mathrm{LGr}(2,4)$ are dual, yielding a self-dual sequence
$$0\rightarrow F^*\rightarrow \C^4\otimes\O\rightarrow F\rightarrow 0.$$
Thus the pair of bundles on $X$ degenerate to isomorphic bundles in this case, as pulling back $E^*\cong F$ to $\tilde{Z}_p$, and then restricting to $X\subset\tilde{Z}_p$, yields isomorphic bundles. We conclude that the double cover $S$ of $\P^2$ branched over $C$ naturally parametrizes a family of stable bundles on $X$ with Mukai vectors $v=(2,h,2)$. Now let us show that every bundle in $M_X(2,h,2)$ arises in this way.

\vspace*{2mm}
\noindent
{\bf Claim:} A stable bundle $\mathcal{E}$ with Mukai vector $v(\mathcal{E})=(2,h,2)$ satisfies
$$h^i(\mathcal{E})=\mathrm{dim}\H^i(X,\mathcal{E})=\left\{\begin{array}{lr} 4 & \mbox{if }i=0, \\ 0 & \mbox{otherwise.} \end{array}\right.$$

\begin{prf}
This follows from standard arguments involving stable sheaves. Firstly, $\H^2(X,\mathcal{E})\cong\H^0(X,\mathcal{E}^{\vee})^{\vee}$ vanishes because $\mathcal{E}^{\vee}$ is also a stable bundle, with slope $\mu(\mathcal{E}^{\vee})=-h^2/2=-4$. Next, suppose that $\H^1(X,\mathcal{E})$ is non-vanishing. Then $\mathrm{Ext}^1(\O,\mathcal{E}^{\vee})=\H^1(X,\mathcal{E}^{\vee})\cong\H^1(X,\mathcal{E})^{\vee}$ is also non-vanishing, so there is a non-trivial extension
$$0\rightarrow \mathcal{E}^{\vee}\rightarrow \mathcal{F}\rightarrow\O\rightarrow 0.$$
Now $\mathcal{F}$ has Mukai vector $(3,-h,3)$ and
$$(3,-h,3)^2=(-h)^2-2.3.3=-10<-2,$$
so $\mathcal{F}$ cannot be stable. Let $\mathcal{G}\subset\mathcal{F}$ be a destabilizing sheaf; then $\mathcal{G}$ has slope
$$\mu(\mathcal{G})=\frac{c_1(\mathcal{G})\cdot h}{r(\mathcal{G})}\geq \mu(\mathcal{F})=\frac{-h^2}{3}.$$
Moreover, $\mathcal{G}$ is necessarily of rank $1$ or $2$, so writing $c_1(\mathcal{G})=dh$ with $d\in\mathbb{Z}$, we find that $d\geq -r(\mathcal{G})/3>-1$. Therefore $d\geq 0$ and $\mu(\mathcal{G})\geq 0$. Let $g$ be the composition
$$\mathcal{G}\rightarrow\mathcal{F}\rightarrow\O.$$
The kernel of $g$ is then a subsheaf of $\mathcal{E}^{\vee}$ with slope $\mu(\mathrm{ker}g)\geq 0$; by the stability of $\mathcal{E}^{\vee}$, $\mathrm{ker}g$ must vanish. So $\mathcal{G}\cong\O$ and $\mathcal{G}\rightarrow\mathcal{F}$ gives a splitting of the exact sequence defining $\mathcal{F}$, contradicting the fact that the extension class is non-trivial. We conclude that $\H^1(X,\mathcal{E})$ must vanish. Finally, Riemann-Roch gives
$$\chi(\mathcal{E})=\int_X(2,h,2)(1,0,1)=4,$$
so $h^0(\mathcal{E})=4$. This completes the proof of the claim.
\end{prf}

It follows from the claim that $\mathcal{E}$ has precisely four independent sections. Moreover, we can show that $\mathcal{E}$ is generated by its sections, i.e., the evaluation map
$$\H^0(X,\mathcal{E})\otimes\O\longrightarrow\mathcal{E}$$
is surjective. Roughly, if the evaluation map were not surjective, it would factor through
$$\H^0(X,\mathcal{E})\otimes\O\longrightarrow\mathcal{E}^{\prime}\longrightarrow\mathcal{E}.$$
One can then argue that $\mathcal{E}^{\prime}$ has Mukai vector $(2,h,2-k)$ where $k\geq 1$, and hence the kernel $\mathcal{F}$ in
$$\mathcal{F}\longrightarrow\H^0(X,\mathcal{E})\otimes\O\longrightarrow\mathcal{E}^{\prime}$$
will have Mukai vector $v(\mathcal{F})=(2,-h,2+k)$. Since $v(\mathcal{F})^2=-4k<-2$, $\mathcal{F}$ will be unstable. Looking at the slope, we see that $\mathcal{F}$ must have a section. The composition
$$\O\longrightarrow\mathcal{F}\longrightarrow\H^0(X,\mathcal{E})\otimes\O$$
will identify $\O$ isomorphically with its image in $\H^0(X,\mathcal{E})\otimes\O$, which will look like $\langle s\rangle\otimes\O$ for some non-zero section $s\in\H^0(X,\mathcal{E})$. But then $\langle s\rangle\otimes\O$ will lie in the kernel of the evaluation map $\H^0(X,\mathcal{E})\otimes\O\rightarrow\mathcal{E}$. This is only possible if $s=0$, a contradiction.

Thus every stable sheaf $\mathcal{E}$ with Mukai vector $(2,h,2)$ is naturally a quotient of the trivial rank four bundle, implying that there is a classifying map $X\rightarrow\mathrm{Gr}(2,4)$ such that $\mathcal{E}$ is the pullback of the universal quotient bundle on $\mathrm{Gr}(2,4)$. Generically, the classifying map will be an embedding and compatible with the embeddings into $\P^5$, and thus $\mathrm{Gr}(2,4)$ can be identified with a smooth quadric four-fold containing $X$, i.e., we have
$$X\subset\mathrm{Gr}(2,4)\cong Z_p\subset\P^5$$
for some $p\in\P^2\backslash C$. But then $\mathcal{E}$ belongs to the family of bundles on $X$ parametrized by $S$, as described above. Note that the covering involution is given by mapping $\mathcal{E}$ to the cokernel of the adjoint map
$$\mathcal{E}^*\longrightarrow\H^0(X,\mathcal{E})^*\otimes\O.$$

In the non-generic case, the stable bundle $\mathcal{E}$ obtained in this way fits into a self-dual sequence
$$\mathcal{E}^*\longrightarrow\H^0(X,\mathcal{E})^*\otimes\O\cong\H^0(X,\mathcal{E})\otimes\O\longrightarrow\mathcal{E}.$$
In particular, there is a skew two-form on $\H^0(X,\mathcal{E})$, and the classifying map factors through the lagrangian Grassmannian
$$X\longrightarrow\mathrm{LGr}(2,4)\subset\mathrm{Gr}(2,4);$$
it is no longer an embedding. The lagrangian Grassmannian is a hyperplane section of the usual Grassmannian, $\mathrm{LGr}(2,4)\subset\P^4$, and the required singular quadric $Z_p$ containing $X$ is a cone over $\mathrm{LGr}(2,4)$.
\end{prf}

\begin{rmk}
We could instead observe that $S$ parametrizes a complete family of stable sheaves on $X$, which therefore must be all of $M_X(v)$, since the latter is two-dimensional.
\end{rmk}

\begin{lem}
\label{lem: BSvariety}
The K3 surface $S$ comes with a Brauer-Severi variety $W\rightarrow S$ whose fibres are isomorphic to $\P^3$.
\end{lem}

\begin{prf}
We have a family $Z\rightarrow\P^2$ of quadric fourfolds over $\P^2$. Let $W\rightarrow\P^2$ be the Fano variety of maximal isotropic planes ($\cong\P^2$) contained in the fibres of $Z\rightarrow\P^2$. If $p\in\P^2\backslash C$ then $Z_p$ is a smooth quadric fourfold, which therefore contains two $\P^3$-families of maximal isotropic planes. Therefore the fibre $W_p$ will consist of two copies of $\P^3$. If $p\in C$ then $Z_p$ is a singular quadric fourfold, and the above familes degenerate to a single $\P^3$-family of maximal isotropic planes. Therefore the fibre $W_p$ will be a single $\P^3$.

The morphism $W\rightarrow\P^2$ therefore factors through the double cover of $\P^2$ branched over $C$, i.e., it factors through $W\rightarrow S\rightarrow\P^2$ (this is the Stein factorization). Then $W\rightarrow S$ is the required Brauer-Severi variety, with fibres isomorphic to $\P^3$; see Proposition~3.3 of \cite{hvv11}.
\end{prf}

\begin{lem}
\label{lem: BSorderTwo}
The Brauer-Severi variety $W\rightarrow S$ gives a class $\alpha\in\Br(S)$ in the Brauer group of $S$ of order two.
\end{lem}

\begin{prf}
A priori, the order of the class $\alpha$ must divide four, as locally $W$ is the projectivization of a rank four bundle. By Proposition~B.3 of Auel, Bernardara, and Bolognesi~\cite{abb11} $\alpha$ is also the class arising from the even Clifford algebra on the discriminant cover $S\rightarrow\P^2$. Specifically, a $\P^2$-family of quadrics in $\P^5$ is equivalent to a quadratic form $q$ in six variables over the field $\C(\P^2)$. Because the rank of $q$ is even, the corresponding Clifford algebra $C(q)$ is a central simple algebra over $\C(\P^2)$, whereas the centre of the even Clifford algebra $C_0(q)$ is the quadratic extension $\C(S)$ of $\C(\P^2)$ given by adjoining the square root of the sextic discriminant; $C_0(q)$ is then a central simple algebra over $\C(S)$ (see Lam~\cite{lam05}). The result of Auel et al.\ identifies $\alpha\in\Br(S)\subset\Br(\C(S))$ with the Brauer class of $C_0(q)$.

Now the Clifford algebra $C(q)$ admits a canonical involution $\sigma$ sending $x_1\otimes\cdots\otimes x_k$ to $x_k\otimes\cdots\otimes x_1$. This anti-automorphism induces an automorphism of Azumaya algebras
\begin{eqnarray*}
C(q)\otimes C(q) & \longrightarrow & \mathrm{End}(C(q)) \\
x\otimes y & \longmapsto & (z\mapsto xz\sigma(y)),
\end{eqnarray*}
implying that the Brauer class of $C(q)$ in $\Br(\C(\P^2))$ has order two. Finally, the Brauer class of $C_0(q)$ is the pull-back of the Brauer class of $C(q)$ to $\C(S)$. To see this, let $V$ be the underlying six-dimensional vector space of the quadratic form $q$. Then the map
\begin{eqnarray*}
V\otimes_{\C(\P^2)}\C(S) & \longrightarrow & \mathrm{End}_{C_0(q)}(C_0(q)\oplus C_1(q)) \\
v & \longmapsto & \left(\begin{array}{cc} 0 & v \\ v & 0 \end{array}\right)
\end{eqnarray*}
induces the required isomorphism of Azumaya algebras
$$C(q)\otimes_{\C(\P^2)}\C(S)\stackrel{\cong}{\longrightarrow} \mathrm{End}_{C_0(q)}(C_0(q)\oplus C_1(q))\cong C_0(q)\otimes_{\C(\P^2)}\mathrm{M}_{2\times 2}(\C(\P^2))$$
by the universal property of the Clifford algebra. It follows that the Brauer class of $C_0(q)$, and hence $\alpha\in\Br(S)$, also has order two.
\end{prf}

\begin{rmk}
It is not obvious from the Clifford algebra description that $\alpha\in\Br(S)$ is non-trivial, but this follows from the following lemma, since we saw in Section~\ref{ss: mukai set-up} that $S$ is a non-fine moduli space when $n:=\mathrm{gcd}(2,h^2,2)=2$ is greater than $1$.
\end{rmk}

\begin{lem}
\label{obstruction}
The class $\alpha\in\Br(S)$ of the Brauer-Severi variety $W\rightarrow S$ is the obstruction to the existence of a universal sheaf for the moduli space $S=M_X(2,h,2)$.
\end{lem}

\begin{prf}
Universal sheaves exist locally, so let $S=\cup_iS_i$ be a cover such that there exists a local universal sheaf $\mathcal{U}_i$ on each $X\times S_i$. Denote by $p$ and $q$ the projections from $X\times S$ to $X$ and $S$, respectively. On the overlap
$$(X\times S_i)\cap (X\times S_j)$$
$\mathcal{U}_i$ and $\mathcal{U}_j$ will differ by tensoring with $q^*\mathcal{L}_{ij}$, where $\mathcal{L}_{ij}$ is a line bundle on $S_{ij}:=S_i\cap S_j$. The collection of line bundles $\mathcal{L}_{ij}$ defines a holomorphic gerbe on $S$, whose Brauer class is the obstruction to the existence of a universal sheaf on $X\times S$.

By the claim in the proof above, $\H^0(X,\mathcal{U}_i|_{X\times \{s\}})$ is four-dimensional for all $s\in S_i$. Therefore $q_*\mathcal{U}_i$ is a locally free sheaf of rank four on $S_i$. Moreover,
$$q_*\mathcal{U}_j=q_*(q^*\mathcal{L}_{ij}\otimes\mathcal{U}_i)=\mathcal{L}_{ij}\otimes q_*\mathcal{U}_i.$$
Therefore the local $\P^3$-bundles $\P(q_*\mathcal{U}_i)$ patch together to give a globally defined $\P^3$-bundle on $S$.

\vspace*{2mm}
\noindent
{\bf Claim:} This $\P^3$-bundle can be identified with the Brauer-Severi variety $W\rightarrow S$.

\begin{prf}
Let $\mathcal{E}:=\mathcal{U}_i|_{X\times \{s\}}$. Recall that $\mathcal{E}$ is realized as a quotient
$$\H^0(X,\mathcal{E})\otimes\O\longrightarrow\mathcal{E},$$
which is the pullback of
$$\C^4\otimes\O\longrightarrow F$$
by the classifying map $X\rightarrow\mathrm{Gr}(2,4)$. Therefore a line in $\H^0(X,\mathcal{E})$ corresponds to a line $\ell$ in $\C^4$. But each line $\ell$ in $\C^4$ determines a maximal isotropic plane in $\mathrm{Gr}(2,4)$, namely, the set of planes in $\C^4$ containing $\ell$ is isomorphic to $\P(\C^4/\ell)\cong\P^2$. Thus the family of lines in $\H^0(X,\mathcal{E})$ gives one half of the Fano variety of maximal isotropic planes in $\mathrm{Gr}(2,4)$, parametrized by $\P(\C^4)=\P^3$.

To get the other half, recall that the covering involution of $S\rightarrow\P^2$ takes $s$ to the point representing the cokernel $\mathcal{E}^{\prime}$ of
$$\mathcal{E}^*\longrightarrow\H^0(X,\mathcal{E})^*\otimes\O.$$
For this sheaf $\mathcal{E}^{\prime}$, a line in $\H^0(X,\mathcal{E}^{\prime})=\H^0(X,\mathcal{E})^*$ will correspond to a line $\ell$ in $(\C^4)^*$, or equivalently, a hyperplane $\ell^{\perp}$ in $\C^4$. Each hyperplane $\ell^{\perp}$ in $\C^4$ determines a maximal isotropic plane in $\mathrm{Gr}(2,4)$, namely, the set of planes in $\C^4$ contained in $\ell^{\perp}$ is isomorphic to $\P(\ell^{\perp})\cong\P^2$. This gives the other half of the Fano variety of maximal isotropic planes in $\mathrm{Gr}(2,4)$, parametrized by $\P(\C^4)^*=(\P^3)^*$.

In summary, if the points $s$ and $s^{\prime}\in S$ sitting above $p\in\P^2$ represent sheaves $\mathcal{E}$ and $\mathcal{E}^{\prime}$ on $X$, then the family of lines in $\H^0(X,\mathcal{E})$ and $\H^0(X,\mathcal{E}^{\prime})$ can be identified with $W_p$, the Fano variety of maximal isotropic planes contained in the quadric $Z_p\cong\mathrm{Gr}(2,4)$. But this implies that the $\P^3$-bundle on $S$ given locally by $\P(q_*\mathcal{U}_i)$ is precisely the Brauer-Severi variety $W\rightarrow S$, proving the claim.
\end{prf} 

It follows from the claim that if there exists a universal sheaf $\mathcal{U}$ on $X\times S$ then the Brauer-Severi variety $W\rightarrow S$ is the projectivization of the rank four bundle $q_*\mathcal{U}$, and hence the Brauer class $\alpha$ is trivial. Conversely, if $\alpha$ is trivial then the $\P^3$-bundle is the projectivization of a rank four bundle $\mathcal{V}$ on $S$. Moreover, $\mathcal{V}$ must be locally isomorphic to $q_*\mathcal{U}_i$, i.e., it must be {\em equal\/} to $\mathcal{M}_i\otimes q_*\mathcal{U}_i$ for some line bundle $\mathcal{M}_i$ on $S_i$. Then the local universal sheaves $q^*\mathcal{M}_i\otimes\mathcal{U}_i$ on $X\times S_i$ will patch together to give a global universal sheaf $\mathcal{U}$ on $X\times S$.
\end{prf}

\subsection{The degree eighteen/degree two duality}
\label{ss: 18to2}

The fact that general K3 surfaces of degrees four, six, and eight are complete intersections is classical. Mukai~\cite{mukai88} extended this analysis by showing that K3 surfaces of degrees ten to eighteen are linear sections of homogeneous varieties. In particular, a general degree eighteen K3 surface $X$ is a linear section of a certain homogeneous variety $Y:=G_2/P$. This homogenous variety $Y$ is five-dimensional and embeds in $\P(V)=\P^{13}$ (here $V$ is the adjoint representation of $G_2$, and $Y$ is the orbit of the maximal weight vector). The K3 surface $X$ will be the intersection of $Y\subset\P^{13}$ with a codimension three linear subspace $\P(U)=\P^{10}$.

As before, we projectively dualize. The dual variety $\check{Y}$ is again a sextic hypersurface in $\check{\P}^{13}=\P(V^*)$ and $\P(U^{\perp})=\P^2$ intersects this hypersurface in a plane sextic curve $C$. So once again the projective dual of $X$ is a degree two K3 surface $S$, the double cover of $\P^2$ branched over $C$. The geometry of this projective duality was studied extensively by Kapustka and Ranestad~\cite{kr10}, and we shall use their results below.

\begin{lem}
The K3 surface $S$ can be naturally identified with the moduli space $M_X(3,h,3)$.
\end{lem}

\begin{prf}
This is Theorem~1.2 of~\cite{kr10}. A point $p$ in $\P^2\subset\check{\P}^{13}$ corresponds to a hyperplane $H_p$ in $\P^{13}$, which intersects $Y\subset\P^{13}$ in a Fano fourfold $Z_p$ of genus ten and index two. The picture is:
$$\begin{array}{ccccccc}
 & & & \mbox{dually} & & & \\
Y=G_2/P & \hookrightarrow & {\P}(V)={\P}^{13} & & \check{Y} & \hookrightarrow & {\P}(V^*)=\check{\P}^{13} \\
\cup & & \cup & & \cup & & \cup \\
Z_p & \hookrightarrow & H_p={\P}^{12} & & C & \hookrightarrow & {\P}(U^{\perp})={\P}^2 \\
\cup & & \cup & & & & \cup \\
X & \hookrightarrow & {\P}(U)={\P}^{10} & & & & p \\
\end{array}$$
If $p\in\P^2\backslash C$, then $Z_p$ is smooth. Kuznetsov~\cite{kuznetsov06} showed that $Z_p$ admits a pair of vector bundles of rank three, each with six independent sections. This result was clarified by Kapustka and Ranestad, who showed that $Z_p$ admits a unique embedding as a linear section of the Grassmannian $\mathrm{Gr}(3,6)$, up to automorphisms of $\mathrm{Gr}(3,6)$ of course. Denoting the universal bundle and universal quotient bundle on the Grassmannian by $E$ and $F$ respectively, the required rank three bundles on $Z_p$ are precisely the restrictions of $E^*$ and $F$. Further restricting the bundles to $X\subset Z_p$ yields two stable vector bundles on $X$ with Mukai vectors $v=(3,h,3)$. 

If $p\in C$, then $Z_p$ is singular and the pair of bundles on $X$ degenerate to isomorphic bundles in this case. Thus the double cover $S$ of $\P^2$ branched over $C$ naturally parametrizes a family of stable bundles on $X$ with Mukai vectors $v=(3,h,3)$. Since this is a complete family, and the moduli space $M_X(v)$ is two-dimensional, we conclude that $S\cong M_X(3,h,3)$.
\end{prf}

\begin{lem}
The K3 surface $S$ comes with a Brauer-Severi variety $W\rightarrow S$ whose fibres are isomorphic to $\P^2$.
\end{lem}

\begin{prf}
We have a family $Z\rightarrow\P^2$ of Fano fourfolds of genus ten and index two over $\P^2$. Let $W\rightarrow\P^2$ be the Fano variety of cubic surface scrolls contained in the fibres of $Z\rightarrow\P^2$. If $p\in\P^2\backslash C$ then $Z_p$ is smooth and Proposition~1.5 of~\cite{kr10} states that there are two disjoint $\P^2$-families of cubic surface scrolls on $Z_p$. Therefore the fibre $W_p$ will consist of two copies of $\P^2$. If $p\in C$ then $Z_p$ is singular, the above families degenerate to a single $\P^2$-family of cubic surface scrolls, and the fibre $W_p$ is a single $\P^2$.

The morphism $W\rightarrow\P^2$ therefore factors through the double cover of $\P^2$ branched over $C$, i.e., it factors through $W\rightarrow S\rightarrow\P^2$ (this is the Stein factorization). Then $W\rightarrow S$ is the required Brauer-Severi variety, with fibres isomorphic to $\P^2$.
\end{prf}

\begin{rmk}
The Brauer-Severi variety $W\rightarrow S$ gives a class $\alpha$ in the Brauer group of $S$ whose order divides three. If $\alpha$ is non-trivial, it will therefore be $3$-torsion. Non-triviality will follow from the next lemma.
\end{rmk}

\begin{lem}
The class $\alpha\in\Br(S)$ of the Brauer-Severi variety $W\rightarrow S$ is the obstruction to the existence of a universal sheaf for the moduli space $S=M_X(3,h,3)$.
\end{lem}

\begin{prf}
As in Lemma~\ref{obstruction}, we let $S=\cup_iS_i$ be a cover such that there exists a local universal sheaf $\mathcal{U}_i$ on each $X\times S_i$. These local universal sheaves will differ by tensoring with $q^*\mathcal{L}_{ij}$, where $q:X\times S\rightarrow S$ is projection to the second factor, and the collection of line bundles $\mathcal{L}_{ij}$ on $S_i\cap S_j$ will define a holomorphic gerbe on $S$, whose Brauer class obstructs the existence of a universal sheaf on $X\times S$.

Applying the same argument as earlier, one can show that $\mathrm{H}^0(X,\mathcal{U}_i|_{X\times\{s\}})$ is six-dimensional for all $s\in S_i$. Therefore $q_*\mathcal{U}_i$ is a locally free sheaf of rank six on $S_i$. Moreover, $q_*\mathcal{U}_j=\mathcal{L}_{ij}\otimes q_*\mathcal{U}_i$, so the local $\mathbb{P}^5$-bundles $\mathbb{P}(q_*\mathcal{U}_i)$ patch together to give a global $\mathbb{P}^5$-bundle on $S$.

\vspace*{2mm}
\noindent
{\bf Claim:} This $\P^5$-bundle can be identified with the second symmetric power $\mathrm{Sym}^2W\rightarrow S$ of the Brauer-Severi variety $W\rightarrow S$.

\begin{rmk}
We follow the convention that applying $\mathrm{Sym}^2$ to a projective space means taking $\mathrm{Sym}^2$ of the underlying vector space, then projectivizing. Thus $\mathrm{Sym}^2W\rightarrow S$ denotes the Brauer-Severi variety that is locally given by $\mathbb{P}(\mathrm{Sym}^2E_i)\rightarrow S_i$, where $E_i\rightarrow S_i$ are rank three bundles such that $W|_{S_i}\cong\mathbb{P}(E_i)$.

In fact, this operation can be applied directly to the corresponding Azumaya algebra. In Section~5 of~\cite{suslin91}, Suslin described how to construct exterior powers $\lambda^iA$ of an Azumaya algebra $A$, an idea that was further developed by Parimala and Sridharan~\cite{ps94}. The symmetric powers $s^iA$ of an Azumaya algebra can be constructed in a similar way; for example, see Section~3.A of Knus et al.'s book~\cite{kmrt98}.
\end{rmk}

\begin{prf}
The bundle $\mathcal{E}:=\mathcal{U}_i|_{X\times\{s\}}$ is realized as a quotient
$$\mathrm{H}^0(X,\mathcal{E})\otimes\mathcal{O}\longrightarrow\mathcal{E},$$
which is the pullback of
$$\mathbb{C}^6\otimes\mathcal{O}\longrightarrow F$$
by the classifying map $X\hookrightarrow Z_p\hookrightarrow\mathrm{Gr}(3,6)$. Therefore a line in $\mathrm{H}^0(X,\mathcal{E})$ corresponds to a line $\ell$ in $\mathbb{C}^6$. The set of $3$-planes in $\mathbb{C}^6$ containing $\ell$ then determines a subvariety $T_{\ell}\subset\mathrm{Gr}(3,6)$ of codimension three, isomorphic to $\mathrm{Gr}(2,5)$. For a generic line $\ell$, $T_{\ell}\cap Z_p$ will be a curve. However, for some choices of $\ell$, the intersection $T_{\ell}\cap Z_p$ is not transversal; instead, $T_{\ell}\cap Z_p$ is a cubic surface scroll (two-dimensional). Moreover, the set
$$W^{(1)}_p:=\{\ell\subset\mathbb{C}^6|T_{\ell}\cap Z_p\mbox{ is a cubic surface scroll}\}$$
is isomorphic to $\mathbb{P}^2$, embedded as a Veronese surface in the space $\mathbb{P}^5$ of all lines in $\mathbb{C}^6$ (see Section~3 of~\cite{kr10}, particularly Proposition~3.13). Since the Veronese embedding is given by the second symmetric power,
$$\mathbb{P}(\mathbb{C}^3)\hookrightarrow\mathbb{P}(\mathrm{Sym}^2\mathbb{C}^3),$$
we see that the family of lines in $\mathrm{H}^0(X,\mathcal{E})$ can be canonically identified with $\mathrm{Sym}^2W^{(1)}_p$, where $W^{(1)}_p\cong\mathbb{P}^2$ is one half of the Fano variety $W_p$ of cubic surface scrolls in $Z_p$.

To recover the other half of the Fano variety $W_p$, we consider instead hyperplanes in $\mathrm{H}^0(X,\mathcal{E})$. These correspond to hyperplanes $\ell^{\perp}$ in $\mathbb{C}^6$, which determine subvarieties $T_{\ell^{\perp}}\subset\mathrm{Gr}(3,6)$, again isomorphic to $\mathrm{Gr}(3,5)\cong\mathrm{Gr}(2,5)$, parametrizing $3$-planes in $\ell^{\perp}$. The set
$$W^{(2)}_p:=\{\ell^{\perp}\subset\mathbb{C}^6|T_{\ell^{\perp}}\cap Z_p\mbox{ is a cubic surface scroll}\}$$
is isomorphic to $\mathbb{P}^2$, again embedded as a Veronese surface in the space $(\mathbb{P}^5)^*$ of all hyperplanes in $\mathbb{C}^6$. Thus the family of hyperplanes in $\mathrm{H}^0(X,\mathcal{E})$ can be canonically identified with $\mathrm{Sym}^2W^{(2)}_p$, where $W^{(2)}_p$ is the other half of the Fano variety $W_p$ of cubic surface scrolls in $Z_p$. 

Finally, we observe that the covering involution of $S\rightarrow\mathbb{P}^2$ takes the point $s$ representing $\mathcal{E}$ to the point representing the cokernel $\mathcal{E}^{\prime}$ of
$$\mathcal{E}^*\longrightarrow\mathrm{H}^0(X,\mathcal{E})^*\otimes\mathcal{O}.$$
Then a hyperplane in $\mathrm{H}^0(X,\mathcal{E})$ will correspond to a line in $\mathrm{H}^0(X,\mathcal{E}^{\prime})\cong\mathrm{H}^0(X,\mathcal{E})^*$.

Altogether, we have shown that the $\mathbb{P}^5$-bundle on $S$ given locally by $\mathbb{P}(q_*\mathcal{U}_i)$ is precisely the Brauer-Severi variety $\mathrm{Sym}^2W\rightarrow S$, proving the claim.
\end{prf}

The symmetric power $s^2A$ of an Azumaya algebra is Brauer equivalent to $A\otimes_k A$ (see page~33 of~\cite{kmrt98}). Equivalently, the Brauer class of $\mathrm{Sym}^2W\rightarrow S$ is the same as the Brauer class of the tensor product $\otimes^2W\rightarrow S$, which is given by $\alpha^2$. If there exists a universal sheaf $\mathcal{U}$ on $X\times S$, then the claim implies that the Brauer-Severi variety $\mathrm{Sym}^2W\rightarrow S$ will be the projectivization of the rank six bundle $q_*\mathcal{U}$, and hence its Brauer class $\alpha^2$ will be trivial. Since the order of $\alpha$ divides three, we conclude that $\alpha$ is trivial.

Conversely, if $\alpha$ is trivial then the $\mathbb{P}^5$-bundle is the projectivization of a rank six bundle $\mathcal{V}$ on $S$. Moreover, $\mathcal{V}$ is locally isomorphic to $q_*\mathcal{U}_i$, i.e., {\em equal\/} to $\mathcal{M}_i\otimes q_*\mathcal{U}_i$ for some line bundle $\mathcal{M}_i$ on $S_i$. The local universal sheaves $q^*\mathcal{M}_i\otimes\mathcal{U}_i$ will then patch together to give a global universal sheaf $\mathcal{U}$ on $X\times S$.
\end{prf}

\subsection{The degree sixteen/degree four duality}
\label{ss: 16to4}

By Mukai's results~\cite{mukai88} every degree sixteen K3 surface $X$ is a linear section of the Lagrangian Grassmannian $Y:=\mathrm{LGr}(3,6)$. The homogeneous variety $Y$ is six-dimensional and embeds in $\P(V)=\P^{13}$. The K3 surface $X$ will be the intersection of $Y\subset\P^{13}$ with a codimension four linear subspace $\P(U)=\P^9$.

As before, we projectively dualize. The dual variety $\check{Y}$ is a quartic hypersurface in $\check{\P}^{13}=\P(V^*)$ and $\P(U^{\perp})=\P^3$ intersects this hypersurface in a quartic K3 surface $S$. This is the projective dual of $X$. It was studied by Iliev and Ranestad~\cite{ir05, ir07}.

\begin{lem}
The K3 surface $S$ can be naturally identified with the moduli space $M_X(2,h,4)$.
\end{lem}

\begin{prf}
This is Theorem~3.4.8 of~\cite{ir05}. A point $p$ in $\check{Y}$ corresponds to a hyperplane $H_p$ in $\P^{13}$ that is tangent to $Y\subset\P^{13}$. In particular, if $p\in S\subset\check{Y}$, then $H_p$ intersects $Y$ in a singular fivefold $Z_p$. The picture is:
$$\begin{array}{ccccccc}
 & & & \mbox{dually} & & & \\
Y=\mathrm{LGr}(3,6) & \hookrightarrow & {\P}(V)={\P}^{13} & & \check{Y} & \hookrightarrow & {\P}(V^*)=\check{\P}^{13} \\
\cup & & \cup & & \cup & & \cup \\
Z_p & \hookrightarrow & H_p={\P}^{12} & & S & \hookrightarrow & {\P}(U^{\perp})={\P}^3 \\
\cup & & \cup & & \cup & & \\
X & \hookrightarrow & {\P}(U)={\P}^9 & & p & & \\
\end{array}$$
The five-fold $Z_p$ will have a single node at the point of tangency of $H_p$ and $Y$. Projecting from this node yields an embedding of the blow-up $\tilde{Z}_p$ in $\P^{11}$. Note that $Z_p$ is degree sixteen in $\P^{12}$, whereas $\tilde{Z}_p$ will be degree fourteen in $\P^{11}$. In fact, by Theorem~3.3.4 of~\cite{ir05}, $\tilde{Z}_p$ embeds as a linear section of the Grassmannian $\mathrm{Gr}(2,6)$, which itself embeds as a degree fourteen subvariety of $\P^{14}$. The picture is:
$$\begin{array}{ccc}
\mathrm{Gr}(2,6) & \hookrightarrow & {\P}^{14} \\
\cup & & \cup \\
\tilde{Z}_p & \hookrightarrow & {\P}^{11} \\
\end{array}$$
The node of $Z_p$ does not lie on the K3 surface $X$, so that the embedding $X\subset Z_p$ lifts to an embedding $X\subset \tilde{Z}_p$. Composing this with the embedding $\tilde{Z}_p\subset\mathrm{Gr}(2,6)$ yields an embedding of $X$ in the Grassmannian. We then obtain a rank two bundle on $X$ by restricting the dual $E^*$ of the universal bundle of the Grassmannian. Iliev and Ranestad prove that this vector bundle on $X$ is stable with Mukai vector $v=(2,h,4)$.

Thus the quartic K3 surface $S$ naturally parametrizes a family of stable bundles on $X$ with Mukai vectors $v=(2,h,4)$. Since this is a complete family, and the moduli space $M_X(v)$ is two-dimensional, we conclude that $S\cong M_X(2,h,4)$.
\end{prf}

\begin{que}
The moduli space $M_X(2,h,4)$ is not fine. Rather, there is a $2$-torsion Brauer class on the K3 surface $S=M_X(2,h,4)$ obstructing the existence of a universal sheaf. Can a Brauer-Severi variety $W\rightarrow S$ representing this Brauer class be described in a natural way? Does it have fibres isomorphic to $\P^1$ or to $\P^3$?
\end{que}

\begin{rmk}
Kuznetsov also studied projective duality for the Lagrangian Grassmannian $\mathrm{LGr}(3,6)$, in Section~7 of~\cite{kuznetsov06}. He constructed a conic bundle over the smooth part of the quartic hypersurface $\check{Y}$ (Lemma~7.8~\cite{kuznetsov06}). This $\P^1$-bundle is a Brauer-Severi variety representing the Brauer class of a certain Azumaya algebra on (the smooth part of) $\check{Y}$ (Proposition~7.9~\cite{kuznetsov06}). When restricted to $S\subset\check{Y}$, presumably this Brauer class gives the obstruction to a universal sheaf for the moduli space $S=M_X(2,h,4)$, and the conic bundle gives $W\rightarrow S$. Assuming this is true, we would still like to find an interpretation of $W\rightarrow S$ in terms of Fano varieties of the hyperplane sections $Z_p$, as in the previous examples.
\end{rmk}

\section{Other dualities}

Other examples of projective dualities do not appear to lead to Brauer elements on K3 surfaces. A K3 surface $X$ of degree fourteen embeds as a linear section of the Grassmannian $Y:=\mathrm{Gr}(2,6)\subset\P^{14}$. The dual variety $\check{Y}\subset\check{\P}^{14}$ is a cubic hypersurface, and the projective dual of $X$ is a Pfaffian cubic fourfold $F$. The picture is:

$$\begin{array}{ccccccc}
 & & & \mbox{dually} & & & \\
Y=\mathrm{Gr}(2,6) & \hookrightarrow & {\P}^{14} & & \check{Y} & \hookrightarrow & \check{\P}^{14} \\

\cup & & \cup & & \cup & & \cup \\
X & \hookrightarrow & {\P}^8 & & F & \hookrightarrow & {\P}^5 \\
\end{array}$$
Beauville and Donagi~\cite{bd85} showed that the Fano variety of lines on a (general) cubic fourfold is a four-dimensional holomorphic symplectic manifold, and in particular, the Fano variety of lines on the Pfaffian cubic $F$ is isomorphic to the Hilbert scheme $\mathrm{Hilb}^{[2]}X$ of two points on $X$. We can write $\mathrm{Hilb}^{[2]}X$ as a Mukai moduli space $M_X(1,0,-1)$. Unfortunately it is a fine moduli space, so it does not come with a Brauer element.

Iliev and Ranestad~\cite{ir01} associated a second cubic fourfold to a degree fourteen K3 surface $X$, which they called the {\em apolar cubic\/}. The Fano variety of lines on the apolar cubic parametrizes presentations of the Pfaffian cubic as a sum of ten cubes. This Fano variety is also isomorphic to $\mathrm{Hilb}^{[2]}X$, though with a different polarization. So again there is no Brauer element; in any case, the apolar cubic does not arise from projective duality.

Another duality, studied by Iliev and Markushevich~\cite{im04}, is between pairs of degree twelve K3 surfaces. A K3 surface $X$ of degree twelve is a linear section of a ten-dimensional spinor variety in $\P^{15}$. The projective dual is another K3 surface $S$ of degree twelve, which can be identified with $M_X(2,h,3)$. However, since this is a fine moduli space, it does not come with a Brauer element.

There are also interesting derived equivalences between these dual varieties. The appropriate machinery is Kuznetsov's Homological Projective Duality, and some of these examples are studied from that point of view in~\cite{kuznetsov06,kuznetsov07}.

\section{Application: Failure of weak approximation}
\label{S: Failure of WA}

	Let $X$ be a K3 surface of degree eight over a number field $k$, given as a complete intersection of three quadrics in $\P^5$
	\[
		X = V\left(Q_1, Q_2, Q_3\right) \subseteq \P^5 = \Proj k[x_0,\dots,x_5].
	\]
	In this section we give an explicit description, in terms of quaternion algebras over function fields, for the class $\alpha \in \Br(S)$ that obstructs the existence of a universal sheaf on the Mukai moduli space $S = M_X(2,h,2)$ described in~\S\ref{ss: 8to2}. We then use the description of $\alpha$ to exhibit K3 surfaces $S$ of degree two that fail to satisfy weak approximation on account of $\alpha$, via a Brauer-Manin obstruction.
	
	 The incidence correspondence 
	 \[
	 	Z := \{ xQ_1 + yQ_2 + zQ_3 = 0\} \subset \P^2\times \P^5,
	 \]
	 has $X$ as it base locus, and projection map $Z \to \P^2$ is a family of quadric fourfolds.  Let $W \to \P^2$ be the Fano variety of maximal isotropic planes contained in the fibers of $Z \to \P^2$. We saw in Lemmas~\ref{lem: BSvariety} and~\ref{lem: BSorderTwo} (and their proofs) that the Stein factorization $W \to S \to \P^2$ consists of the discriminant cover $S \to \P^2$ and a Brauer-Severi variety $W \to S$ that is \'etale-locally a $\P^3$-bundle over $S$.  The image of the corresponding Brauer class $\alpha \in \Br(S)$ in $\Br \left(k(S)\right)$ is thus an algebra of degree $4$ and exponent $2$; a result of Albert ensures that it is Brauer equivalent to a bi-quaternion algebra; see~\cite{albert}.  To compute this biquaternion algebra, we use the interpretation of $\alpha \in \Br(S)$ as the class associated to the even Clifford algebra on the discriminant cover $S \to \P^2$ (see the proof of Lemma~\ref{lem: BSorderTwo}).
	
	For $i = 1$, $2$ and $3$, let $M_i$ denote the Gram symmetric matrix associated to the quadric $Q_i$.  The Gram symmetric matrix of the quadratic form $q$ in six variables associated to $S \to \P^2$ is $M(x,y,z) := xM_0 + yM_1 + zM_2$, and the signed discriminant of $q$ is 
	\[
		\Delta := - \det(xM_0 + yM_1 + zM_2).
	\]
	Thus, we may write $S$ as the surface in $\P(1,1,1,3) = \Proj k[x,y,z,w]$ given by
	\begin{equation}
	\label{eq: S}
		w^2 = - \det\left(M(x,y,z)\right)
	\end{equation}
	 The discriminant algebra $k(\P^2)(\sqrt{\Delta})$ is the function field $k(S)$. To compute $\alpha$ as the class in $\im\left(\Br(S) \to \Br\left(k(S)\right)\right)$ of the even Clifford algebra $C_0(q)$, we recall some facts about quadratic forms.
	 
	 \medskip
	 
	 \noindent{\bf Notation} Given nonzero elements $a$ and $b$ of a field $K$, write $(a,b)$ for the quaternion algebra which, as a four-dimensional $K$-vector space, is spanned by $1$, $i$, $j$, and $ij$, with multiplication determined by the relations $i^2 = a$, $j^2 = b$ and $ij = -ji$. Abusing notation, we sometimes also denote by $(a,b)$ the class of the quaternion algebra in $\Br(K)$.

\subsection{Quadratic forms of rank 6 and the even Clifford algebras}
	
	Let $q$ be a nondegenerate quadratic form of even rank over a field $K$ of characteristic not two.  Let $\Delta$ be the signed discriminant of $q$, and let $L = K\left(\sqrt{\Delta}\right)$ be the discriminant extension (we assume that $\Delta$ is not a square in $K$). Write $c(q) \in \Br(K)$ (resp.\ $c_0(q) \in \Br(L)$) for the class of the Clifford algebra $C(q)$ (resp.\ the even Clifford algebra $C_0(q)$). A straightforward generalization of the last part of the proof of Lemma~\ref{lem: BSorderTwo} establishes the following lemma.
	
	\begin{lem}
		\label{lem: evencliff}
		We have $c_0(q) = c(q)\otimes_K L$ as classes in $\Br(L)$.
		\qed
	\end{lem}
			
	\begin{lem}
		\label{lem: clifford}
		Let $a \in K^*$, and write $\langle a \rangle$ for the rank one quadratic form $aX^2$. Let $q$, $q_1$, and $q_2$ be nondegenerate quadratic forms of even rank, with respective signed discriminants $\Delta$, $\Delta_1$, and $\Delta_2$.
		\begin{enumerate}
		\item[(i)] $c(\langle a \rangle \otimes q) = c(q) \otimes (a,\Delta)$
		\item[(ii)] $c(q_1 \perp q_2) = c(q_1) \otimes c(q_2) \otimes (\Delta_1,\Delta_2)$
		\item[(iii)] $c(q \perp \langle a,-a\rangle) = c(q)$
		\end{enumerate}
	\end{lem}
	
	\begin{prf}
		Items (i) and (ii) follow from Proposition~IV.8.1.1 of \cite{knus}. For (iii), recall that $c(\langle a,b \rangle) = (a,b)$, so that by (ii) we have $c(q \perp \langle a,-a\rangle) = c(q) \otimes (a,-a) \otimes (1,\Delta)$ and both $(a,-a)$ and $(1,\Delta)$ are trivial in $\Br(K)$.
	\end{prf}
	
	We specialize to quadratic forms $q$ of rank six; diagonalizing, we may assume that $q = \langle a_1,\dots,a_6\rangle$ for some $a_i \in k$, $i = 1,\dots, 6$. Lemma~\ref{lem: clifford}(iii) allows us to add hyperbolic planes to $q$ without so changing the class of $c(q)$. Consider the quadratic form
	\[
		\langle a_1,a_2,a_3,a_4,a_5,a_6\rangle \perp \langle a_1a_2a_3,-a_1a_2a_3\rangle
		= \langle a_1,a_2,a_3,a_1a_2a_3\rangle \perp \langle a_4,a_5,a_6,-a_1a_2a_3\rangle,
	\]
	which is equivalent to the sum
	\begin{equation}
	\label{eq: decomp}
	\langle a_1a_2a_3\rangle\otimes \langle 1, a_1a_2,a_2a_3,a_1a_3\rangle
		\perp \langle -a_1a_2a_3\rangle\otimes \langle 1,-a_1a_2a_3a_4,-a_1a_2a_3a_5,-a_1a_2a_3a_6\rangle
	\end{equation}
	The quadratic form $\langle 1, a_1a_2,a_2a_3,a_1a_3\rangle$ is the norm form of the quaternion algebra $(-a_1a_2,-a_1a_3)$. Applying Lemma~\ref{lem: clifford} to the forms $q_1 = \langle 1, a_1a_2\rangle$ and $q_2 = \langle a_2a_3,a_1a_3\rangle$ we obtain
	\[
		c(\langle 1, a_1a_2,a_2a_3,a_1a_3\rangle) = (-a_1a_2,-a_1a_3).
	\]
	Applying Lemma~\ref{lem: clifford} to~\eqref{eq: decomp} we compute the class of $c(q)$ and obtain
	\[
		(-a_1a_2,-a_1a_3)\otimes c(\langle 1,-a_1a_2a_3a_4,-a_1a_2a_3a_5,-a_1a_2a_3a_6\rangle)  \otimes (-a_1a_2a_3,\Delta(q)) \in \Br(k)
	\]
	Over the discriminant extension $L = K\left(\sqrt{\Delta}\right)$, the quaternion algebra $(-a_1a_2a_3,\Delta)$ splits, and we have an equivalence of quadratic forms
	\[
		\langle 1,-a_1a_2a_3a_4,-a_1a_2a_3a_5,-a_1a_2a_3a_6\rangle \cong \langle 1,-a_1a_2a_3a_4,-a_1a_2a_3a_5,a_4a_5\rangle,
	\]
	the latter of which is the norm form of the quaternion algebra $(a_1a_2a_3a_4,a_1a_2a_3a_5)$.  Putting this all together, we obtain
	\[
		c(q)\otimes_K L = (-a_1a_2,-a_1a_3) \otimes (a_1a_2a_3a_4,a_1a_2a_3a_5) \in \Br L
	\]
	Lemma~\ref{lem: evencliff} then allows us to conclude the following.
	
	\begin{prp}
		\label{prop: evencliff computation}
		Let $q = \langle a_1,\dots,a_6\rangle$ be a nondegenerate diagonal quadratic form of rank six over a field $K$ of characteristic not two, with nontrivial discriminant extension $L$. Then
		\[
			c_0(q) = (-a_1a_2,-a_1a_3) \otimes (a_1a_2a_3a_4,a_1a_2a_3a_5) \in \Br(L) \eqno\qed
		\]
	\end{prp}
		
	\begin{cor}
		\label{cor:handy}
		Let $q$ be a nondegenerate quadratic form of rank 6 over a field $K$ of characteristic not two, with nontrivial discriminant extension $L$. Write $m_i$ for the determinant of the leading principal $i \times i$ minor of the Gram symmetric matrix of $q$. Then
		\[
			c_0(q) = (-m_2, -m_1 m_3)\otimes (m_4,-m_3m_5) \in \Br(L).
		\]
\end{cor}

	\begin{prf}
		Symmetric Gaussian elimination of $M$ allows us to diagonalize $M$ to the matrix
		\[
			\text{diag}({ m_1, m_2/m_1, \dotsc, m_{6}/m_{5} });
		\]
		See the proof of Lemma~12 in \cite{abbv} for  details on this operation. Proposition~\ref{prop: evencliff computation} implies that
		\[
			c_0(q) = (-m_2,-m_1m_3/m_2) \otimes (m_4,m_3m_5/m_4)
		\]
		Finally, we have the equalities of classes in $\Br(L)$
		\[
			(-m_2,-m_1m_3/m_2) = (-m_2,-m_1m_2m_3) = (-m_2,m_2)\otimes (-m_2,-m_1m_3) = (-m_2,-m_1m_3)
		\]
		and similarly
		\[
			(m_4,m_3m_5/m_4) = (m_4,-m_3m_5).\eqno\qed
		\]
		\hideqed
\end{prf}

\subsection{Brauer-Manin obstructions}
\label{S:BMobs}

Let $S$ be a smooth projective geometrically integral variety over a number field $k$.  Write $\kbar$ for a fixed algebraic closure of $k$, and let $\Sbar$ denote the fibered product $S\times_k \kbar$. Write $\Adeles$ for the ring of adeles of $k$, and $\Omega$ for the set of places of $k$. Since $S$ is projective, the sets $S(\Adeles)$ and $\prod_{v \in \Omega} S(k_v)$ coincide; here $k_v$ denotes the completion of $k$ at $v \in \Omega$. A class $\calC$ of varieties as above is said to satisfy the {\defi Hasse principle} if
\[
S(\Adeles) \neq \emptyset \implies S(k) \neq \emptyset \qquad \textup{for every }S \in \calC.
\]
We say that $S$ satisfies \defi{weak approximation} if the diagonal embedding of $S(k)$ in $\prod_{v \in \Omega} S(k_v) = S(\Adeles)$ is dense for the product topology of the $v$-adic topologies.

Manin used class field theory to observe that any subset $\calS$ of the Brauer group $\Br(S) = \H_{\et}^2(S,\G_m)$ gives rise to an intermediate set
\begin{equation}
\label{eq:obs}
\overline{S(k)} \subseteq S(\Adeles)^\calS \subseteq S(\Adeles),
\end{equation}
where $\overline{S(k)}$ denotes the closure of $S(k)$ in $S(\Adeles)$; see~\cite{manin}. These intermediate sets can thus obstruct the Hasse principle (if $S(\Adeles) \neq \emptyset$ yet $S(\Adeles)^\calS = \emptyset$), and weak approximation (if $S(\Adeles) \neq S(\Adeles)^\calS$). This kind of obstruction is known as a \defi{Brauer-Manin obstruction}.

For each $x_v \in S(k_v)$, there is an evaluation map $\Br(S) \to \Br(k_v)$, $\alpha \mapsto \alpha(x_v)$ obtained by applying the functor $\H^2_{\et}(-,\G_m)$ to the morphism $\Spec k_v \to S$ corresponding to $x_v$. The set $S(\Adeles)^\calS$ is the intersection over $\alpha \in \calS$ of the sets
\[
	S(\Adeles)^\alpha := \bigg\{ (x_v) \in S(\Adeles) : \sum_{v \in \Omega} \inv_v\left(\alpha(x_v)\right) = 0\bigg\};
\]
here $\inv_v \colon \Br(k_v) \to \Q/\Z$ is the local invariant map at $v$ from local class field theory.

There is a filtration on the Brauer group
\[
\Br_0(S) \subseteq \Br_1(S) \subseteq \Br(S)
\]
where $\Br_0(S) := \im\left(\Br(\Spec k) \to \Br (S)\right)$ is the subgroup of \defi{constant} Brauer elements, and $\Br_1(S) := \ker\left(\Br(S) \to \Br(\Sbar)\right)$ is the subgroup of \defi{algebraic} Brauer elements. Classes $\alpha \in \Br(S)\setminus\Br_1(S)$ are called \defi{transcendental}.

K3 surfaces are some of the simplest varieties on which transcendental classes exist: curves and surfaces of negative Kodaira dimension have trivial geometric Brauer groups. For example, if $S$ is a K3 surface with $S(\Adeles) \neq \emptyset$ and $\Pic(\Sbar) = \NS(\Sbar) \cong \Z$, then any nonconstant class in $\Br(S)$ is transcendental, because there is an isomorphism
\[
	\Br_1(S)/\Br_0(S) \xrightarrow{\sim} \H^1(\Gal(\kbar/k),\Pic(\Sbar)),
\]
coming from the Hochschild-Serre spectral sequence, and the group $\H^1(\Gal(\kbar/k),\Pic(\Sbar))$ is trivial because $\Pic(\Sbar)$ is free with trivial Galois action in this case.

Details for the material in this subsection can be found in the surveys \cite{peyre,VA} and Chapter~5 of Skorobogatov's book \cite{skoro}.

\subsection{Local invariants at the real place}

We return to the situation at the beginning of \S\ref{S: Failure of WA}, specializing to the case where $k = \Q$. So let $X$ be a K3 surface of degree eight over $\Q$, and let $S \subseteq \P(1,1,1,3)$ be the associated degree two K3 surface, together with the class $\alpha \in \Br(S)$. Write, as before, $M(x,y,z)$ for the Gram symmetric matrix of the quadratic form $q$ in six variables associated to $S \to \P^2$. Let $P_0 = [x_0, y_0, z_0, w_0] \in S(\R)$ be a real point of $S$. From~\eqref{eq: S} it follows that $\det\left(M(x_0, y_0, z_0)\right) < 0$, so the signature of the symmetric matrix $M(x_0, y_0, z_0)$ is $(1, 5)$, $(5,1)$ or $(3,3)$. 
	
\begin{lem}
         \label{lemm:signature}
         Write $\infty$ for the real place of $\Q$. We have
         \[
			 \inv_\infty\left(\alpha(P_0)\right) =
			 \begin{cases}
			 0 & \textrm{if }\Sign(M(x_0,y_0,z_0)) = (3,3), \\
			 \frac{1}{2} & \textrm{if }\Sign(M(x_0,y_0,z_0)) = (1,5) \textrm{ or } (5,1).
			 \end{cases}
         \]
\end{lem}
	
	\begin{proof}
		This is an application of Proposition \ref{prop: evencliff computation}, noting that $q$ has coefficients in $K := \Q(\P^2)$, and that the discriminant extension $L$ is $\Q(S)$. The proof of Lemma~\ref{lem: BSorderTwo} shows that $\alpha = c_0(q)$ in $\Br(\Q(S))$. We deduce from Proposition~\ref{prop: evencliff computation} that
		\begin{equation}
		\label{eq:realpts}
		\alpha = (-a_1a_2,-a_1a_3) \otimes (a_1a_2a_3a_4,a_1a_2a_3a_5) \in \Br(\Q(S))
		\end{equation}
		We may now compute invariants for the specialization $\alpha(P_0)$. For example, suppose that $\Sign(M(x_0,y_0,z_0)) = (3,3)$. Without loss of generality, we may assume that $a_i(P_0)$ is positive for $i = 1, 2$ and $3$, and negative for $i = 4, 5$, and $6$. Then all the entries of the quaternion algebras in~\eqref{eq:realpts} are negative, and hence $\alpha(P_0) = 0$ as an element of $\Br(\R)$. Consequently, $\inv_\infty\left(\alpha(P_0)\right) = 0$.  The other possible signatures for $M(x_0,y_0,z_0)$ are handled similarly.
	\end{proof}
	
	\begin{cor}
	         \label{cor:continuity}
	         Suppose that $S(\R)\neq \emptyset$.  Then there exists a point $P\in S(\R)$ such that
	         \[
	            \inv_\infty\left(\alpha(P)\right) = 0.
		\]
	\end{cor}
	\begin{proof}
	The set $S(\R)$ of real points of $S$ is a 2-dimensional real manifold since it is not empty; see~\cite{shafarevich}, p.\ 106.  Hence there is a point $P = [x_0,y_0,z_0,w_0] \in S(\R)$ such that $\det(M(x_0,y_0,z_0)) \neq 0$, because the set of real points of the discriminant curve in $S$ is either empty or has real dimension $1$.  Consider the signed projective plane $\mathbb S = \R^3 \setminus \{(0,0,0)\} / \R_{>0}$, which is topologically a sphere. Write $Q = [x_0,y_0,z_0] \in \mathbb S$ and ${-Q} = [-x_0,-y_0,-z_0] \in \mathbb S$ for the two points in $\mathbb S$ corresponding to $P$.  If the signature of $M(Q)$ is $(3,3)$, then by Lemma~\ref{lemm:signature} we are done.  Since the signature of $M(-Q)$ is negative that of $M(Q)$, we may assume that $M(Q)$ has signature $(1,5)$ and $M(-Q)$ has signature $(5,1)$. 
	
	Let $\gamma$ denote the discriminant curve in $\mathbb S$, i.e.,
	\[
	 \gamma = \{ [x,y,z] \in \mathbb S \mid \det ( M(x,y,z)) = 0\},
	\]
	which is a disjoint union of smooth closed curves. We claim there is a line $\ell\subset \mathbb S$ that contains $Q$ and $-Q$, and that meets $\gamma$ transversally.  This is an application of Bertini's theorem: Consider $Q$ as a point in $\P^2$, and note that the set of lines passing through $Q$ is a $\P^1$. Bertini assures us that (over $\C$) the set of lines through $Q$ meeting the curve $\det(M(x,y,z)) = 0$ in $\P^2$ forms a nonempty open subset $U\subseteq \P^1$. Hence $U(\R) = U\cap \P^1(\R) \neq \emptyset$, and any element of this set gives a line $\ell\subset \mathbb S$, as desired.
	
	Let $f$ be the restriction of $\det ( M(x,y,z))$ to $\ell$. Then $f$ has simple roots by transversality of $\gamma\cap \ell$.
	This implies that as we travel from $Q$ to $-Q$ along $\ell$ and cross $\gamma$, the signature of $M(x,y,z)$ will change from $(a,b)$ to either $(a+1,b-1)$ or $(a-1,b+1)$, and starting from signature $(1,5)$, we must reach signature $(5,1)$. Hence, along $\ell$, there must be a point $R \in \mathbb S$ such that the signature of $M(R)$ is $(3,3)$, and consequently $\det(M(R)) < 0$. Lifting $R$ to a point in $S(\R)$, and applying Lemma~\ref{lemm:signature}, we obtain the desired result.
	\end{proof}

\subsection{An explicit example}

	Let $X$ be the K3 surface of degree eight over $\Q$ given as the smooth complete intersection of three quadrics in $\P^5$ with Gram matrices
	{\small
	\[
		M_1 :=
		\begin{pmatrix}
			-6 & 1 & -3 & 3 & -1 & 1 \\
			1 & 26 & 3 & 2 & 2 & 3 \\
			-3 & 3 & 2 & 1 & 2 & -3 \\
			3 & 2 & 1 & 28 & 0 & 0 \\
			-1 & 2 & 2 & 0 & 12 & 1 \\
			1 & 3 & -3 & 0 & 1 & 8
		\end{pmatrix},
		\quad
		M_2 :=
		\begin{pmatrix}
			0 & -1 & 0 & -3 & 1 & 3 \\
			-1 & 8 & 1 & -2 & -2 & 3 \\
			0 & 1 & 24 & 2 & -3 & -3 \\
			-3 & -2 & 2 & -2 & -1 & -2 \\
			1 & -2 & -3 & -1 & 28 & 3 \\
			3 & 3 & -3 & -2 & 3 & 16
		\end{pmatrix},
	\]

	\[
		\textup{and}\quad M_3 :=
		\begin{pmatrix}
			8 & 2 & -1 & -2 & 0 & 0 \\
			2 & 32 & 0 & 0 & -3 & -2 \\
			-1 & 0 & 8 & -1 & 3 & 0 \\
			-2 & 0 & -1 & 24 & -3 & -1 \\
			0 & -3 & 3 & -3 & 28 & 3 \\
			0 & -2 & 0 & -1 & 3 & 32
		\end{pmatrix}.
	\]
	}

	\begin{prp}
		\label{Prop: Pic1}
		Let $S$ be the K3 surface of degree two in $\P(1,1,1,3) = \Proj\Q[x,y,z,w]$ given by
	\[
		w^2 = -\det\left(x M_1 + y M_2 + z M_3\right)
	\]
	Then $\Pic(\Sbar) = \NS(\Sbar) \cong \Z$.
	\end{prp}
	
	\begin{proof}
		We follow the strategy of \S5.3 of \cite{hv11}, which can be summarized as follows:
		\begin{enumerate}
			\item Prove that $\Pic\left( S\times \overline{\F}_3 \right)$ is isomorphic to $\Z^2$, generated by the two components of the pullback from $\P^2$ of a tritangent line to the sextic branch curve.
			\item Find a prime $p > 3$ of good reduction of $S$ for which the reduced sextic branch curve has no tritangent line.
			\item Apply Proposition~5.3 of~\cite{hv11} to conclude that $\Pic(\Sbar) \cong \Z$: otherwise the tritangent line over $\F_3$ would lift to $\Q$, giving rise to a tritangent line over $\F_p$ for any other prime $p$ of good reduction for $S$.
		\end{enumerate}
		
		The surface $S$ has good reduction at $3$. An equation for $S\times\F_3$ is given by
		\[
			\begin{split}
			w^2 &= (x + 2z)(x^4y + 2x^3y^2 + 2x^3z^2 + 2x^2y^3 + x^2z^3 + 2xy^4 \\
        &\qquad +xy^3z + 2xy^2z^2 + xyz^3 + 2y^4z + y^3z^2 +
        2z^5) + (x^2y + y^3)^2,
			\end{split}
		\]
		from which it is clear that the line $x + 2z = 0$ is tritangent to the branch sextic on $\P^2$. The pullback of this tritangent line to $S\times {\F}_3$ generates a rank two sublattice of $\Pic\left( S\times \overline{\F}_3 \right)$. Let $f$ be the characteristic polynomial for the action of Frobenius on ${\rm H}^2_\textrm{\'et}\left( S \times \overline{F}_3, \overline{\Q}_\ell\right)$, where $\ell \neq 3$ is a prime number. Normalize this polynomial by setting $f_3(t) = 3^{-22}f(3t)$. Then the rank of $\Pic\left( S\times \overline{\F}_3 \right)$ is bounded above by the number of roots of $f_3(t)$ that are roots of unity; see Corollary~2.3 of \cite{vanLuijk}. The computation of $f_3(t)$ is standard: it suffices to determine $\#S(\F_{3^n})$ for $n = 1,\dots,10$; the Lefschetz trace formula and the functional equation for $f$ then allows one to determine enough traces of powers of Frobenius acting on ${\rm H}^2_\textrm{\'et}\left( S \times \overline{F}_3, \overline{\Q}_\ell\right)$ to reconstruct $f$ by elementary linear algebra. See \cite{vanLuijk} for details. We obtain
		\[
			\begin{split}
				f_3(t) = \frac{1}{3}(t - 1)^2&(3t^{20} + t^{19} + 2t^{18} + t^{17} + 3t^{16} + t^{15} + 2t^{14} - t^{13} - t^{12} \\
				&\qquad - t^{11} - t^9 - t^8 - t^7 + 2t^6 + t^5 + 3t^4 + t^3 + 2t^2 + t + 3)
			\end{split}
		\]
		The roots of the degree $20$ factor of $f_3(t)$ are not roots of unity, because they are not integral. Hence $\Pic\left( S\times \overline{\F}_3 \right) \cong \Z^2$.
		
		A Gr\"obner basis computation, using \cite[Algorithm~8]{EJ-ANTS}, shows that the reduction of $S$ at $5$ (which is smooth) has no line tritangent to the branch curve. This concludes the proof of the proposition.
	\end{proof}
	
	By Corollary~\ref{cor:handy}, the Brauer class $\alpha \in \Br(S)$ arising from $X$ is represented over $\Br\left(\Q(S)\right)$ by tensor product of quaternion algebras $(-m_2, -m_1 m_3)\otimes (m_4,-m_3m_5)$, where
	\begin{equation}
	\label{eq:mms}
	\begin{split}
    		m_1 &= -6x + 8z, \\
		m_2 &= -157x^2 - 46xy + 12xz - y^2 + 68yz + 252z^2, \\
		m_3 &= -512x^3 - 3884x^2y - 1790x^2z - 1094xy^2 - 48xyz \\
			&\qquad + 370xz^2 - 24y^3 + 1618y^2z + 6580yz^2 + 1984z^3, \\
    		m_4 &= -14896x^4 - 112256x^3y - 64196x^3z - 13639x^2y^2 - 88686x^2yz \\
			&\qquad - 31415x^2z^2 + 1230xy^3 + 28380xy^2z +
    190454xyz^2 + 66580xz^3 \\
    			&\qquad - 1967y^4 - 14274y^3z +
    12573y^2z^2 + 148652yz^3 + 46212z^4, \\
    		m_5 &= -154622x^5 - 1832494x^4y - 1088428x^4z - 3261270x^3y^2 \\
			&\qquad - 6264622x^3yz - 2086758x^3z^2 - 353890x^2y^3 -
    2306720x^2y^2z \\
    			&\qquad - 992652x^2yz^2 - 124086x^2z^3 + 2698xy^4 + 587200xy^3z \\
			&\qquad + 6271452xy^2z^2 + 9184426xyz^3 + 2279020xz^4 - 51948y^5 \\
			&\qquad - 439790y^4z - 82534y^3z^2 +
    4374124y^2z^3 + 5413502yz^4 + 1214952z^5
    \end{split}
   \end{equation}
    Consider the real points on $S$ given by
    \[
    		P_1 := [1,2,-1,924]\qquad\textrm{and}\qquad
		P_2 := [0,-1,1,\sqrt{1863673}]
    \]
    Using Lemma~\ref{lemm:signature}, we compute
    \begin{equation}
    \label{eq:invariants}
    		\inv_\infty\left(\alpha(P_1)\right) = 0, \qquad\textrm{and}\qquad
		\inv_\infty\left(\alpha(P_2)\right) = \frac{1}{2}.
    \end{equation}
    The point $P_1$, embedded diagonally in $S(\Adeles)$, lies in the set $S(\Adeles)^\alpha$; see~\eqref{eq:obs}. Let $(P_v) \in S(\Adeles)$ be the adelic point given by
    \[
    		P_v =
		\begin{cases}
			P_1 & \textrm{if } v \neq \infty, \\
			P_2 & \textrm{otherwise},
		\end{cases}
    \]
    The containment $P_1 \in S(\Adeles)^\alpha$ and~\eqref{eq:invariants} together imply that
    \[
	    	\sum_v \inv_v \alpha(P_v) = \frac{1}{2} \in \Q/\Z.
    \]
    Hence $(P_v) \in S(\Adeles)\setminus S(\Adeles)^{\alpha}$, which shows that $S$ is does not satisfy weak approximation on account of $\alpha$. As explained in \S\ref{S:BMobs}, Proposition~\ref{Prop: Pic1} implies that
    \[
    		{\rm H}^1\left(\Gal(\Qbar/\Q),\Pic(\Sbar)\right) = 0,
    \]
    so there is no {\em algebraic} Brauer-Manin obstruction to weak approximation on $S$.  We summarize our results in the following theorem.

	\begin{thm}
		\label{thm:counterexample}
		Let $M_1$, $M_2$ and $M_3$ be the three symmetric matrices defined above. Let $S$ be the K3 surface of degree two in $\P(1,1,1,3) = \Proj\Q[x,y,z,w]$ given by
	\[
		w^2 = -\det\left(x M_1 + y M_2 + z M_3\right)
	\]
	Let $\alpha\in \Br\left(\Q(S)\right)$ be the tensor product of quaternion algebras $(-m_2, -m_1 m_3)\otimes (m_4,-m_3m_5)$, with $m_1,\dots,m_5$ as in~\eqref{eq:mms}. Then $\alpha$ extends to an  element of $\Br(S)$ that gives rise to a transcendental Brauer-Manin obstruction to weak approximation on $S$.
	\end{thm}

\subsection{What about the Hasse principle?}
\label{ss: hp?}

It is natural to ask if  elements $\alpha \in \Br(S)[2]$ as above can obstruct the existence of rational points on $S$.  This does not happen for the surface of Theorem~\ref{thm:counterexample}: the point $P_1$ is rational.

A Brauer-Manin obstruction to the Hasse principle arising from a $2$-torsion Brauer element $\alpha$ requires the image of the evaluation maps
\[
\ev_{\alpha,p}\colon S(\Q_v) \to {\textstyle\frac{1}{2}\Z}/\Z\qquad P \mapsto \inv_v(\alpha(P))
\]
be \emph{constant} for all places $v$, including the infinite place. Otherwise, an adelic point $(P_v) \in S(\Adeles)$ can be modified at a place where $\ev_{\alpha,v}$ is non constant to arrange that $\sum_v \ev_{\alpha,v}(P_v) = 0$, which means that $(P_v) \in S(\Adeles)^\alpha$, so $\alpha$ does not obstruct the Hasse principle.  We must also have $\sum_v \ev_{\alpha,v}(P_v) = \frac{1}{2}$ for every $(P_v) \in S(\Adeles)$.

The evaluation map $\ev_{\alpha,p}$ can only take nonzero values at a finite number of places: the places of bad reduction for $S$, the places where $\alpha$ ramifies, and the infinite place.  To obtain an obstruction to the Hasse principle from $\alpha$, we must have $\ev_{\alpha,\infty}(P) = 0$ for all points $P \in S(\R)$, by Corollary~\ref{cor:continuity}.  We expect that an argument similar to that of Lemma~4.4 of~\cite{hv11}, shows that, for any prime $p\neq 2$ of good reduction for $\alpha$, we have $\ev_{p,\alpha}(P) = 0$ for all $P \in S(\Q_p)$.  For primes $p \neq 2$ of bad reduction, Proposition~4.1 and Lemma~4.2 of \cite{hv11} show that $\ev_{\alpha,p}$ is constant, provided the singular locus of the reduction of $S$ consists of at most 7 ordinary double points. We thus expect that the a reasonable way to construct a counterexample to the Hasse principle using elements of the form $\alpha$ is to pick matrices $M_1$, $M_2$, and $M_3$ in such a way that $\ev_{\alpha, 2}(P) = \frac{1}{2}$ for all points $P \in S(\Q_2)$; an analysis similar to that in Section~4.3 of~\cite{hv11} may prove sufficient for this purpose.

\begin{flushleft}
University of Montana\hfill kelly.mckinnie@umontana.edu\\
Department of Mathematical Sciences\hfill www.math.umt.edu/mckinnie\\
Missoula MT 59812-0864\\
USA\\
\end{flushleft}

\begin{flushleft}
Department of Mathematics\hfill sawon@email.unc.edu\\
University of North Carolina\hfill www.unc.edu/$\sim$sawon\\
Chapel Hill NC 27599-3250\\
USA\\
\end{flushleft}

\begin{flushleft}
Department of Mathematics MS 136\hfill sho.tanomito@rice.edu\\
Rice University\hfill math.rice.edu/$\sim$st26\\
6100 S. Main St.\\
Houston TX 77005-1982\hfill av15@rice.edu\\
USA\hfill math.rice.edu/$\sim$av15\\
\end{flushleft}

\end{document}